\documentclass[fleqn,preprint,3p,a4paper]{elsarticle}
\usepackage{amssymb}
\usepackage{amsmath}
\usepackage{amsthm}
\usepackage{dcolumn}
\usepackage{endnotes}
\usepackage{tabularx}
\usepackage[matrix,arrow]{xy}
\usepackage{wasysym}
\usepackage{graphicx}
\usepackage{epstopdf}
\usepackage{xcolor}
\usepackage{tikz}
\usepackage{mathrsfs}

\newtheorem{theorem}{Theorem}[section]
\newtheorem{proposition}[theorem]{Proposition}
\newtheorem{lemma}[theorem]{Lemma}
\newtheorem{corollary}[theorem]{Corollary}

\theoremstyle{definition}
\newtheorem{definition}[theorem]{Definition}
\newtheorem{example}[theorem]{Example}
\newtheorem{remark}[theorem]{Remark}
\newtheorem{question}[theorem]{Question}

\newcommand{\ir}{{\mathsf{Irr}}}
\newcommand\twoheaduparrow{\mathord{\rotatebox[origin=c]{90}{$\twoheadrightarrow$}}}
\newcommand\twoheaddownarrow{\mathord{\rotatebox[origin=c]{90}{$\twoheadleftarrow$}}}
\newcommand{\dda}{\twoheaddownarrow}
\newcommand{\dua}{\twoheaduparrow}

\newcommand{\cl}{{\rm cl}}
\newcommand{\ii}{{\rm int}}

\newcommand{\ua}{\mathord{\uparrow}}
\newcommand{\da}{\mathord{\downarrow}}

\begin{document}

\begin{frontmatter}

\title{Characterization of $T_0$-spaces for \\ quasi-liminf convergence being topological\tnoteref{t1}}
\tnotetext[t1]{This research was supported by the National Natural Science Foundation of China (Nos. 12471070, 12071199, 12261040).}

\author[X. Wen]{Xinpeng Wen}
\address[X. Wen]{School of Mathematics and Information Science, Nanchang Hangkong University, Nanchang 330063, China}
\ead{wenxinpeng2009@163.com}

\author[X. Xu]{Xiaoquan Xu\corref{mycorrespondingauthor}}
\cortext[mycorrespondingauthor]{Corresponding author}
\ead{xiqxu2002@163.com}
\address[X. Xu]{School of Mathematics and Statistics, Minnan Normal University, Zhangzhou 363000, China}

\begin{abstract}

The authors' primary goal in this paper is to extend some important results related to the liminf-convergence and $\mathcal{QS}$-convergence in domain theory to the setting of $T_0$-spaces. To that end, we study the quasi-liminf convergence in $T_0$-spaces and introduce a new kind of $T_0$-spaces --- \emph{WLH}-spaces. It is proved that every locally hypercompact $T_0$-space is a \emph{WLH}-space, and a $T_0$-space $(X, \tau)$ is a \emph{WLH}-space iff the quasi-liminf convergence in $(X, \tau)$ is topological. Therefore, the quasi-liminf convergence in a locally hypercompact space is topological, and for a quasicontinuous poset $P$, the quasi-liminf convergence is topological and agrees with convergence in the Lawson topology $\lambda(P)$. We also show that a $T_0$-space $(X,\tau)$ is locally hypercompact iff the $\mathcal{QS}$-convergence in $(X,\tau)$ coincides with the convergence in the topology $\tau$. Therefore, a poset $P$ is quasicontinuous iff $\mathcal{QS}$-convergence in the Scott space of $P$ is topological iff $\mathcal{QS}$-convergence coincides with convergence in the Scott topology $\sigma (P)$. Using the quasi-liminf convergence and approximate relation $\ll_c$, we give several characterizations of $C$-spaces and continuous posets.
\end{abstract}

\begin{keyword} $\mathcal{QS}$-convergence; liminf convergence; quasi-liminf convergence; locally hypercompact space; \emph{WLH}-space; $C$-spaces; quasicontinuous poset
\MSC 54A20; 54D45; 06B35; 06F30
\end{keyword}

\end{frontmatter}

\section{Introduction}

In the fields of domain theory and non-Hausdorff topology, a variety of convergent classes have been investigated (see \cite{Andradi-Ho-2019, Chen-Kou-2014, Chen-Kou-2024, Erne-1981, Erne-Gatzke-1985, GHKLMS-2003, Ruan-Xu-2019, Shen-Lu-Li-2021, Yang-Xu-2025, Zhao-Zhao-2005, Zhang-Bao-Xu-2022}). Through diverse convergence structures, not only can numerous concepts of continuity be defined and described, but these structures also connect order and topology in an interconnected way. Within the numerous elegant findings in domain theory, the following two important results serve as a perfect example of this profound link between order theory and topology.

\begin{theorem}\label{theor-quansicontinuous-domain-Lawson-topology} (\cite[Theorem III-3.17]{GHKLMS-2003})  Let $P$ be a dcpo.
\begin{enumerate}[\rm (1)]
\item The Lawson topology and the liminf topology on $P$ agree if $P$ is a quasicontinuous domain.
\item If $P$ is a continuous domain, then the liminf convergence in $P$ is topological and agrees
with convergence in the Lawson topology of $P$.
\end{enumerate}
\end{theorem}

\begin{theorem}\label{theor-charac-quasicontinuous-domain} (\cite[Theorem 2.10]{Chen-Kou-2014}) For a dcpo $P$, the following statements are equivalent:
\begin{enumerate}[\rm (1)]
    \item $P$ is a quasicontinuous domain.

     \item $\mathcal{GS}$-convergence is topological convergence for the Scott topology.
\end{enumerate}
\end{theorem}

At the 6th International Symposium in Domain Theory, Jimmie Lawson emphasized the need to develop the core of
domain theory directly in $T_0$-spaces instead of posets. Towards this direction, several classic results in
domain theory have been lifted from the context of posets to $T_0$-spaces. 


In this paper, we seek to extend the aforementioned two theorems and some other related results to the setting of $T_0$-spaces. The paper is organized as follows.

Section 2 provides some fundamental definitions and notations and a few basic results about lattice-ordered structures and $T_0$-spaces which will be used in the whole paper.

Section 3 is mainly devoted to investigate the $\mathcal{QS}$-convergence in $T_0$-spaces. It is shown that a $T_0$-space $(X,\tau)$ is locally hypercompact iff the $\mathcal{QS}$-convergence in $(X,\tau)$ coincides with the convergence in the topology $\tau$. Therefore, a poset $P$ is quasicontinuous iff $\mathcal{QS}$-convergence in the Scott space of $P$ is topological iff $\mathcal{QS}$-convergence coincides with convergence in the Scott topology $\sigma (P)$.

In Section 4, we study the quasi-liminf convergence in $T_0$-spaces by introducing and investigating a new class of $T_0$-spaces A --- \emph{WLH}-spaces. It is proved that every locally hypercompact $T_0$-space is a \emph{WLH}-space, and a $T_0$-space $(X, \tau)$ is a \emph{WLH}-space iff the quasi-liminf convergence in $(X, \tau)$ is topological. Therefore, the quasi-liminf convergence in a locally hypercompact space is topological, and for a quasicontinuous poset $P$, the quasi-liminf convergence is topological and agrees with convergence in the Lawson topology $\lambda(P)$. Using the quasi-liminf convergence and approximate relation $\ll_c$, we obtain several characterizations of $C$-spaces and continuous posets.

\section{Preliminaries}

In this section, we briefly recall some fundamental definitions and notations and a few basic results about lattice-ordered structures and $T_0$-spaces that will be used in the paper. For further details, we refer the reader to \cite{Engelking-1989, GHKLMS-2003, Goubault-2013, Kelley-1975}.

For a poset $P$ and $A \subseteq P$, define ${\ua} A=\{x\in P: a\leq x \mbox{ for some }a\in A \}$ and ${\da} A=\{x \in P: x \leq a \mbox{ for some }a\in A\}$. For $x\in X$, let ${\ua} x={\ua}\{x\}$ and ${\da} x={\da}\{x\}$. A subset $A$ is called a \emph{lower set} (resp., an \emph{upper set}) if $A={\da} A$ (resp., $A={\ua} A$). The set of all upper bounds (resp., all lower bounds) of $A$ in $P$ is denoted by $A^{\ua}$ (resp., $A^{\da}$), that is, $A^{\ua}=\{u\in P: A\subseteq {\da} u\}$ and $A^{\da}=\{v\in P: A\subseteq \ua v\}$. The set $A^{\delta}=(A^{\ua})^{\da}$ is called the $cut$ of $A$ in $P$. If the set of upper bounds of~$A$ has a unique smallest element, we call this element the \emph{least upper bound} and write it as~$\vee A$ or sup~\!\!$A$ (for supremum). Dually, the \emph{greatest lower bound} of $A$ is written as~$\wedge A$ or inf~\!$A$ (for infimum). A nonempty subset $D$ of a poset $P$ is called \emph{directed} if every nonempty finite subset of $D$ has an upper bound in $D$. The set of all directed sets of $P$ is denoted by $\mathcal{D}(P)$. The poset $P$ is called a \emph{directed complete poset}, or \emph{dcpo} for short, if for any $D\in \mathcal D(P)$, $\vee D$ exists in $P$. A set $I\subseteq P$ is called an \emph{ideal} of $P$ if $I$ is a directed lower subset of $P$. Let $\mathrm{Id} (P)$ be the poset (with the order of set inclusion) of all ideals of $P$. For a nonempty subset $C$ of $P$, define $max(C)=\{c\in C : c \mbox{~ is a maximal element of~} C\}$ and $min(C)=\{c\in A : c \mbox{~ is a minimal element of~} C\}$. The set of all natural numbers is denoted by $\mathbb{N}$. Let $\mathbb N^+=\mathbb N\setminus \{0\}$. A nonempty family $\mathcal{G}$ of subsets of a set $X$ is called \emph{filtered} if for any $G_1, G_2\in \mathcal{G}$, there exists $G_3\in \mathcal{G}$ such that $G_3\subseteq {\uparrow} G_2\cap {\uparrow} G_1$.


\begin{lemma}\label{lem-filtred-family} Suppose that $\{{\ua} F_d : d\in D\}$ is a filtered family of nonempty upper sets of a poset $P$,  and $A, U\subseteq P$. Then we have the following two conclusions:
\begin{enumerate}[\rm (1)]
\item If $F_d\cap {\da} A\neq\emptyset$ for any $d\in D$, then $\{{\ua} (F_d\cap {\da} A) : d\in D\}$ is filtered.
\item If $F_d\not\subseteq {\ua} U$ for any $d\in D$, then $\{{\ua} (F_d\setminus {\ua} U) : d\in D\}$ is filtered.
\end{enumerate}
\end{lemma}
\begin{proof} (1): First, we show that ${\ua} (F\cap {\da} C)={\ua} ({\ua} F\cap {\da} C)$ for any $F, C\subseteq P$. Obviously, ${\ua} (F\cap {\da} C)\subseteq{\ua} ({\ua} F\cap {\da} C)$. Conversely, if $t\in {\ua} ({\ua} F\cap {\da} C)$, then there is $s\in {\ua} F\cap {\da} C$ such that $s\leq t$, whence there exists $f\in F$ with $f\leq s$. So $f\in F\cap {\da} C$ and $f\leq t$. Thus $t\in {\ua} (F\cap {\da} C)$. Therefore, ${\ua} ({\ua} F\cap {\da} C)\subseteq {\ua} (F\cap {\da} C)$, and hence ${\ua} (F\cap {\da} C)={\ua} ({\ua} F\cap {\da} C)$. As $\{{\ua} F_d : d\in D\}$ is filtered, $\{{\ua} (F_d\cap {\da} A) : d\in D\}=\{{\ua} ({\ua} F_d\cap {\da} A) : d\in D\}$ is also filtered.

(2): Let $B=X\setminus {\ua} U$. Then $B={\da} B$ and $F_d\cap {\da} B=F_d\cap B=F_d\cap (X\setminus {\ua} U)=F_d\setminus {\ua} U\neq\emptyset$ for any $d\in D$. Hence by (1) we get that $\{{\ua} (F_d\setminus {\ua} U) : d\in D\}=\{{\ua} (F_d\cap B) : d\in D\}$ is filtered.
\end{proof}

For a $T_0$-space $X$ and $A\subseteq X$, let $\mathcal{O}(X)$ (resp., $\Gamma(X)$) be the set of all open subsets (resp., all closed subsets) of $X$. The closure of a subset $A$ in $X$ is denoted by $\cl_{X}~\!\!A$ or simply by $\cl~\!A$ and the interior of $A$ in $X$ by $\ii_{X}~\!\! A$ or simply by $\ii~\! A$. A nonempty subset $F$ of $X$ is called \emph{irreducible} if for any $C_1, C_2\in\Gamma(X)$, $F\subseteq C_1\cup C_2$ implies $F\subseteq C_1$ or $F\subseteq C_2$. The set of all irreducible subsets of $X$ is denoted by $\ir(X)$. We use $\leq_X$ to denote the \emph{specialization order} of $X$: $x\leq y$ iff $x\in \cl~\! \{y\} $. Clearly, open sets  (resp., closed sets) of $X$ are upper sets (resp., lower sets) of $X$. Let $\mathbf{up}(X)=\{{\uparrow_X} U : U\subseteq X\}$ (that is, the family of all upper subsets of $X$), $X^{(<\omega)}$ be the set of all nonempty finite subsets of $X$ and $\mathbf{Fin}~\!\!X=\{{\ua} F : F\in X^{(<\omega)}\}$. For a subset $M$ of $X$, the \emph{induced topology} or \emph{subspace topology} on $M$ is denoted by $\mathcal O(X)|_M$, that is, $\mathcal O(X)|_M=\{U\cap M : U\in \mathcal O(X)\}$. In what follows, when a $T_0$-space is considered as a poset, the order always refers to the specialization order if no other explanation is given.

A subset $U$ of a poset $P$ is \emph{Scott open} if (i) $U={\uparrow}U$, and (ii) for any directed subset $D$ for which $\vee D$ exists, $\vee D\in U$ implies $D\cap U\neq\emptyset$. The topology formed by all Scott open sets of $P$ is called the \emph{Scott topology}, written as $\sigma(P)$. The space $\Sigma~\!\!P=(P, \sigma(P))$ is called the \emph{Scott space} of $P$. The \emph{lower topology} on a poset $P$, generated by $\{P\setminus {\ua}x : x\in P\}$ (as a subbase), is denoted by $\omega(P)$. Dually, define the \emph{upper topology} on $P$ and denote it by $\upsilon(P)$. The topology generated by $\omega (P)\cup\sigma (P)$ is called the \emph{Lawson topology} on $P$ and is denoted by $\lambda (P)$. The space $\Lambda P=(P, \lambda (P))$ is called the \emph{Lawson space} of $P$. The upper sets of $P$ form the (\emph{upper}) \emph{Alexandroff topology} $\alpha (P)$. For a $T_0$-space $(X, \tau)$, we use the symbol $\omega(X)\vee \tau$ to denote the topology generated
by $\omega (X)\cup\tau$ ($X$ is equipped with the specialization order).


There is more than one way to extend the definition of the Scott topology from
dcpos to general posets. In \cite{Erne-1981, Erne-2009}, Ern\'e
has suggested an alternative approach using cuts of directed sets called the weak Scott
topology or $\sigma_2$-topology. We adopt the terminology from \cite{Lawson-Wu-Xi-2020} and call it the \emph{Scott-Ern\'e topology}.

\begin{definition}\label{def-sigma-2-topology} An upper subset $U$ of a poset $P$ is called $s_2$-\emph{open} if for any directed subset $D$ of $P$, $D^{\delta}\cap U\neq\emptyset$ implies $D\cap U\neq\emptyset$. The collection of all $s_2$-open subsets of $P$ form a topology, the \emph{Scott}-\emph{Ern\'e topology} $\sigma_2(P)$.
\end{definition}

Clearly, $\sigma_2(P)$ is always coarser than $\sigma (P)$, and both topologies coincide on dcpos.

Rudin's Lemma is an important tool in domain theory and non-Hausdorff topology (see \cite{GHKLMS-2003, Gierz-Lawson-Stralka-1983, Goubault-2013}. Rudin \cite{Rudin-1981} proved her lemma by transfinite methods. In this paper we will use the following Jung's version of Rudin's Lemma (see \cite[Theorem 4.11]{Jung-1989} or \cite[Corollary 3.5]{Heckmann-Keimel-2013}).

\begin{lemma}\label{lem-Rudin-lemma} (Rudin's Lemma) Let $P$ be a poset and $\mathcal{F}\subseteq P^{(<\omega)}$ such that $\{{\ua} F : F\in\mathcal F\}$ is a filtered family. Then there exists directed set $D\subseteq \mathop{\bigcup}\limits_{F \in \mathcal{F}} F$ such that $D\cap F\neq\emptyset$ for all $F\in \mathcal{F}$.
\end{lemma}

As an interesting consequence of Rudin's Lemma, we have the following.

\begin{corollary}\label{cor-Rudin-lemma} (\cite[Corollary 3.9]{Heckmann-Keimel-2013}) Let $P$ be a dcpo, $\{{\ua} F_d : d\in D\}\subseteq \mathbf{Fin} P$ be a filtered family and $U\in \sigma(P)$. Then $\bigcap_{d\in D}{\ua}F_d\subseteq U$ implies ${\ua}F_d\subseteq U$ for some $d\in D$.
\end{corollary}



A subset $W$ of a $T_0$-space $X$ is called \emph{saturated} if $W$ equals the intersection of all open sets containing it (equivalently, $W$ is an upper set of $X$). We shall use $Q(X)$ to denote the set of all nonempty compact saturated subsets of $X$ and endow it with the \emph{Smyth order} $\sqsubseteq$: $K_1\sqsubseteq K_2$ if{}f $K_2\subseteq K_1$. The \emph{upper Vietoris topology} on $Q(X)$ is the topology that has $\{\Box U: U\in \mathcal{O}(X)\}$ as a base, where $\Box U=\{Q\in Q(X):Q\subseteq U\}$, and the resulting space is called the \emph{Smyth power space} or \emph{upper space} of $X$ and is denoted by $P_S(X)$. For $A\in \Gamma(X)$, define $\Diamond A=\{Q\in Q(X):Q\cap A\neq\emptyset\}$. Clearly, $Q(X)\setminus \Box U=Q(X)\setminus \Diamond (X\setminus U)$ for any $U\in\mathcal O(X)$.

\begin{remark}\label{rem-Smyth-specialization-order} Let $X$ be a $T_0$-space. Then the specialization order of $P_S(X)$ is the Smyth order, that is, $\leq_{P_S(X)}=\sqsubseteq$. Suppose that $K_1,K_2\in Q(X)$ and $K_1\in \cl_{P_S(X)}\{K_2\}$. If $K_2\not\subseteq K_1$, then there is $k_2\in K_2\setminus K_1$, whence $K_1\in \Box (X\setminus \da k_2)\in \mathcal O(P_S(X))$. Therefore, by $K_1\in \cl_{P_S(X)}\{K_2\}$ we have $K_2\in \Box (X\setminus \da k_2)$ or, equivalently, $K_2\subseteq X\setminus \da k_2$, a contradiction. Thus  $K_2\subseteq K_1$. Conversely, assume $K_2\subseteq K_1$. Then for any $U\in \mathcal O(X)$ with $K_1\in \Box U$, we have $K_2\in \Box U$. Hence $K_1\in \cl_{P_S(X)}\{K_2\}$.
\end{remark}

\begin{lemma}\label{lem-closure-in-Smyth}
Let $(X, \tau)$ be a $T_0$-space and $A\subseteq X$. Then $\cl_{P_S((X, \tau))}\{{\ua} a : a\in A\}=\Diamond \cl_{\tau}~\!\!A$.
\end{lemma}
\begin{proof} Clearly, $\Diamond \cl_{\tau}~\!\!A=Q(X)\setminus \Box (X\setminus \cl_{\tau}~\!\!A)$ is closed in $P_S((X, \tau))$ and hence $\cl_{P_S((X, \tau))}\{{\ua} a : a\in A\}\subseteq\Diamond  \cl_{\tau}~\!\!A$. As $\{\Diamond  C : C\in \Gamma(X)\}$ is a (closed) base of $P_S((X, \tau))$, there is a family $\{C_i : i\in I\}\subseteq \Gamma(X)$ such that $\cl_{P_S((X, \tau))}\{{\ua} a : a\in A\}=\bigcap_{i\in I} \Diamond  C_i$. Then for each $i\in I$, $\{{\ua} a : a\in A\}\subseteq \Diamond  C_i$, and consequently, ${\ua} a\cap C_i\neq \emptyset$ for all  $a\in A$; whence, for each $a\in A$, $a\in C_i$ as $C_i=\da C_i$. It follows that $\cl_{\tau}~\!\!A\subseteq C_i$ for each $i\in I$ and hence $\Diamond  \cl_{\tau}~\!\!A\subseteq \bigcap_{i\in I}\Diamond  C_i=\cl_{P_S((X, \tau))}\{{\ua} a : a\in A\}$. Thus $\cl_{P_S((X, \tau))}\{{\ua} a : a\in A\}=\Diamond \cl_{\tau}~\!\!A$.
\end{proof}

For a poset $P$ and $A, B\subseteq P$, we say $A$ is \emph{way below} $B$, written $A\ll B$, if for each $D\in \mathcal D(P)$ for which $\vee D$ exists in $P$, $\vee D\in \ua B$ implies $D\cap \ua A\neq \emptyset$. For $B=\{x\}$, a singleton, $A\ll B$ is
written $A\ll x$ for short. For $x\in P$, let ${\Downarrow} x=\{u\in P : u\ll x\}$  and $w(x)=\{{\ua} F : F\in P^{(<\omega)}, F\ll x\}$.

\begin{definition}\label{def-continuous-domain} (\cite[Definition 1.2]{Gierz-Lawson-Stralka-1983}, \cite[Definition III-3.2]{GHKLMS-2003}) (1)  A poset (resp., a dcpo) $P$ is called a \emph{continuous poset} (resp., a \emph{continuous domain}), if for each $x\in P$, ${\Downarrow} x$ is directed and $x=\vee {\Downarrow} x$.

(2) A dcpo $Q$ is called a \emph{quasicontinuous domain}, if for each $x\in Q$, $w(x)$ is filtered and ${\ua} x=\bigcap
w(x)$.
\end{definition}

\begin{proposition}\label{prop-quasicontinuous-sufficient-condition} (\cite[Proposition 3.2]{Gierz-Lawson-Stralka-1983}) Let $P$ be a dcpo. Suppose that for each $x\in P$ there exists a filtered subfamily $\mathcal G$ of $w(x)$ satisfying $\bigcap \mathcal G\subseteq {\ua} x$. Then $P$ is a quasicontinuous domain. Furthermore, we have that for a finite set $F$ of $P$, ${\ua} F\in w(x)$ iff ${\ua} G\subseteq {\ua} F$ for some ${\ua} G\in \mathcal G$.
\end{proposition}

As in \cite{Erne-2018}, a $T_0$-space $(X, \tau)$ is called \emph{C-space} if for any $U\in \tau$ and $x\in U$, there exists $y \in X$ with  $x\in \ii_{\tau} \ua y\subseteq \ua y\subseteq U$. The space $(X, \tau)$ is called \emph{locally hypercompact} if for any $U\in \tau$ and  $x\in U$, there exists $F \in X^{(<\omega)}$ such that $x\in \ii_{\tau} \ua F\subseteq  \ua F\subseteq U$.

\begin{proposition}\label{prop-continuous-c-space} (\cite[Theorem II-1.14]{GHKLMS-2003}, \cite[Theorem 4.12]{Xu-LS-2006}) For a poset $P$, the following two conditions are equivalent:
\begin{enumerate}[\rm (1)]
\item $P$ is continuous.
\item $\Sigma~\!\!P$ is a $C$-space.
\end{enumerate}
\end{proposition}

\begin{proposition}\label{prop-quasicontinuous-local-hypercompact} (\cite[Proposition III-3.6 and Exercise III-3.19]{GHKLMS-2003}, \cite[Theorem 4.1]{Mao-Xu-2006}) For a dcpo $P$, the following two conditions are equivalent:
\begin{enumerate}[\rm (1)]
\item $P$ is quasicontinuous.
\item $\Sigma~\!\!P$ is locally hypercompact.
\end{enumerate}
\end{proposition}

\begin{definition}\label{def-continuous-poset} (\cite[Definition 2.1]{Mao-Xu-2006})  A poset $P$ is said to be \emph{quasicontinuous} if $\Sigma~\!\!P$ is locally hypercompact.
\end{definition}

A \emph{net} $(x_i)_{i\in I}$ in a set $X$ is a mapping from a directed set $I$ to $X$.
A net $(y_e)_{e \in E}$ is called a \emph{subnet} of a net $(x_i)_{i\in I}$ if there exists a function $f : E \to I$ such that (i) $y_e = x_{f(e)}$ for all $e \in E$; and (ii) for any $i \in I$, there exists $e \in E$ such that $f(l) \geq i$ whenever $l \geq e$.
  For a property $H(x)$ of elements $x\in X$, we say that $H(x)$ \emph{holds eventually} in the net $(x_i)_{i\in I}$ if there is a $i_0\in I$ such that $H(x_i)$ is true whenever $i_0\leq i$. Similarly, we say that $H(x)$ \emph{holds usually} in the net $(x_i)_{i\in I}$ if for any $i\in I$, there is a $j\in I$ with $i\leq j$ such that $H(x_{j})$ is true.

As we all know, convergence and convergence class play an important role in both order theory and general topology (see \cite{Engelking-1989, GHKLMS-2003, Kelley-1975}). For a topological space $(X, \tau)$ and a class $\mathcal L$ consisting of pairs $((x_i)_{i\in I}, x)$, where $(x_i)_{i\in I}$ is a net in $X$ and $x$ a point of $X$, the topology $\tau$ can naturally induce a convergence class as follows:

\begin{center}
$\mathcal{C}(\tau)=\{((x_i)_{i\in I}, x) : (x_i)_{i\in I} \mbox{ is a net of } X, x\in X \mbox{ and }$ $(x_i)_{i\in I} \mbox{ converges to } x \mbox{ in } (X, \tau)\}$.
\end{center}

\noindent And we can define a topology on $X$ associated with $\mathcal L$:

\begin{center}
$\mathcal O(\mathcal L)=\{U\subseteq X : ((x_i)_{i\in I}, x)\in\mathcal L$ and $x\in U$ imply
 $x_i\in U$ eventually$\}$.
\end{center}

It is easy to verify that $\tau=\mathcal O(\mathcal C(\tau))$. However, if $\mathcal L$ is not a convergence class in the sense of Kelley [2] (see Theorem \ref{theor-convergence-topological} below), then $\mathcal L\neq \mathcal C(\mathcal O(\mathcal L))$, that is, the class $\mathcal L$ is not topological.

\begin{theorem}\label{theor-convergence-topological} (\cite[Exercise II-1.29]{GHKLMS-2003} or \cite[page 74, Theorem 9]{Kelley-1975})
Let $\mathcal{L}$ be a class of some pairs $((x_i)_{i\in I}, x)$ of nets $((x_i)_{i\in I}$ on a set $X$ and elements $x\in X$. Then $((x_i)_{i\in I}, x)\in \mathcal{L}$ iff $((x_i)_{i\in I}$ converges to $x$ with respect to the topology $\mathcal{O}(\mathcal{L})$ precisely when the following four axioms are satisfied.
\begin{description}
    \item[(Constants)] For every constant net $(x)_{i\in I}$ with value $x$ one has $((x)_{i\in I}, x)\in \mathcal{L}$.

    \item[(Subnets)] If $((y_j)_{j\in J}$ is a subnet of $((x_i)_{i\in I}$ and $((x_i)_{i\in I}, x)\in \mathcal{L}$, then $((y_j)_{j\in J}, x)\in \mathcal{L}$.

    \item[(Divergences)] If $((x_i)_{i\in I}, x)$ is not in $\mathcal{L}$, then $((x_i)_{i\in I}$ has a subnet $((y_j)_{j\in J}$ no subnet $((z_k)_{k\in K}$ of which ever has $(((z_k)_{k\in K}, x)\in \mathcal{L}$.

    \item[(Iterated limits)] If $((x_i)_{i\in I}, x)\in \mathcal{L}$, and if $((x_{i,j})_{j\in J(i)}, x_i)\in \mathcal{L}$ for all $i\in I$, then $((x_{i, f(i)})_{(i,f)\in I\times M}, x)\in \mathcal{L}$, where $M=\prod_{i\in I}J(i)=\{f: I\rightarrow \bigcup_{i\in I}J(i): f(i)\in J(i)$ for all $i\in I\}$  is a product of directed sets. The order of $M$ is defined as follows: $f \leq g$ iff $f(i) \leq g(i)$ for all $i\in I$, and the order of $I\times M$ is defined as follows:  $(a, f) \leq (b, g)$   iff $a \leq b$ and $f \leq g$.
\end{description}
\end{theorem}

\begin{lemma}\label{lem-closed-set-lattice-Heyting} (\cite[Theorem 2 and Theorem 3]{Erne-2009})
For a $T_0$-space $(X, \tau)$, the following conditions are equivalent:

\begin{enumerate}[\rm (1)]
\item The lattice of closed sets of $(X, \tau)$ is a complete Heyting algebra.

\item For any $A, B \subseteq X$, $\cl_{\tau}({\da} A\cap {\da} B)=\cl_{\tau}{\da} A \cap \cl_\tau {\da} B$.

\item For each $A\subseteq X$ and $x\in X$, $x \in \cl_{\tau}A$ implies $x\in \cl_{\tau}({\da} x \cap {\da} A)$.

\item For any $U\in \tau$ and $x\in X$, ${\ua} (U\cap {\da} x)\in \tau$.
\item For any $U, V\subseteq X$, $\ii_\tau ({\ua} U\cup {\ua} V)=\ii_\tau {\ua} U\cup \ii_\tau {\ua} V$.
\end{enumerate}
\end{lemma}

As in \cite{Erne-2009}, a $T_0$-space $(X, \tau))$ is called a \emph{web space} if it satisfies the equivalent conditions (1)-(5) of Lemma \ref{lem-closed-set-lattice-Heyting}.

\begin{proposition}\label{prop-C-space-LHC-web} (\cite[Corollary 1]{Erne-2009})
For a $T_0$-space $(X, \tau)$, the following two conditions are equivalent:

\begin{enumerate}[\rm (1)]
\item $(X, \tau)$ is a $C$-space.

\item $(X, \tau)$ is a locally hypercompact web space.
\end{enumerate}
\end{proposition}
\begin{proof} (1) $\Rightarrow$ (2): Obviously, $(X, \tau)$ is locally hypercompact. Let $U, V$ be two subsets of $X$. Then $\ii_\tau \ua U\cup \ii_\tau \ua V \subseteq \ii_\tau (\ua U\cup \ua V)$. Now we show that $\ii_\tau (\ua U\cup \ua V)\subseteq\ii_\tau \ua U\cup \ii_\tau \ua V$. Let $x\in \ii_\tau (\ua U\cup \ua V)$. As $(X, \tau)$ is a $C$-space, there is a $t\in X$ such that $x\in \ii_\tau \ua t\subseteq \ua t\subseteq \ii_\tau (\ua U\cup \ua V)\subseteq \ua U\cup \ua V$. Then $\ua t\subseteq \ua U$ or $\ua t\subseteq \ua V$, whence $x\in \ii_\tau \ua t\subseteq \ii_\tau \ua U$ or $x\in \ii_\tau \ua t\subseteq \ii_\tau \ua V$. So $x\in \ii_\tau \ua U\cup \ii_\tau \ua V$. Thus $\ii_\tau \ua U\cup \ii_\tau \ua V = \ii_\tau (\ua U\cup \ua V)$. By Lemma \ref{lem-closed-set-lattice-Heyting}, $(X, \tau)$ is a web space.

(2) $\Rightarrow$ (1): Assume that $x\in U\in\tau$. Since $(X, \tau)$ is locally hypercompact, there is a finite set $F=\{u_1, u_2, ..., u_n\}\subseteq X$ with $x\in \ii_\tau \ua F\subseteq \ua F\subseteq U$. By Lemma \ref{lem-closed-set-lattice-Heyting}, we have that $x\in \ii_\tau \ua F=\ii_\tau \bigcup\limits_{i=1}^n\ua u_i=\bigcup_{i=1}^n\ii_\tau \ua u_i$. Hence there is $1\leq i\leq n$ such that $x\in \ii_\tau \ua u_i\subseteq \ua u_i\subseteq \ua F\subseteq U$. Thus $(X, \tau)$ is a $C$-space.
\end{proof}

A poset (resp., dcpo, complete lattice) $P$ is said to be a \emph{meet-continuous poset} (resp., \emph{meet-continuous domain}, \emph{meet-continuous lattice}) if for any $x\in P$ and any
directed set $D$ with $x \leq \vee D$, one has $x\in \ii_{\sigma(P)} \downarrow D\cap \downarrow x$.

\begin{proposition}\label{prop-MC-poset-charac} (\cite[Theorem 2.1 and Proposition 2.2]{Kou-Liu-Luo-2003}, \cite[Theorem 3.8]{Mao-Xu-2009}) For a poset $P$, the following two conditions are equivalent:
\begin{enumerate}[\rm (1)]
\item $P$ is meet-continuous.
\item The lattice of Scott-closed sets of $P$ is a complete Heyting algebra (i.e., $\Gamma(\Sigma~\!\!P)$ is a web space).
\end{enumerate}

\indent If $P$ is both a semilattice and a dcpo, then these conditions are equivalent to the following one:

\vspace{0.2cm}

\noindent $\mathrm{(3)}$ $P$ satisfies the following distributive law:
\vskip 3mm
\emph{(MC)}  \qquad \qquad \qquad \qquad \qquad \qquad  $x\wedge (\bigvee D)=\bigvee_{d\in D}x\wedge d$
\vskip 3mm
\noindent for all $x\in P$ and $D\in \mathcal D(P)$.
\end{proposition}

Classically, a subset $U$ of a topological space $X$ is open iff $\overline{A}\cap U\neq\emptyset$ implies $A\cap U \neq\emptyset$ for all subsets $A$ of $X$, or equivalently, every net that converges to a point in $U$ is residually in $U$. In \cite{Erne-2009},
Ern\'e introduced the notion of what he calls a topological space that is ``monotone determined".

\begin{definition}\label{def-MD-space}
A $T_0$-space $(X, \tau)$ is called a \emph{monotone determined space}, an \emph{MD}-\emph{space} for short, if any subset $U$ meeting all directed sets whose closure meets $U$ is open, that is, $U\in\tau$ iff $U\cap\cl_{X} D\neq\emptyset$ implies $U\cap D\neq\emptyset$ for any $D\in \mathcal D(X)$. In this case, the topology $\tau$ is called a \emph{monotone determined topology} (shortly an \emph{MD-topology}). Let $\mathbf{MD}$ denote the category of MD-spaces and continuous mappings.
\end{definition}

 In \cite{Yu-Kou-2015}, Yu and Kou used the convergence of directed subsets to defined the concept of \emph{directed spaces}. For a point $x$ and a directed subset $D$ in a $T_0$-space $X$, it is easy to see that $x\in\cl~\!D$ iff the net $(d)_{d\in D}$ converges to $x$. So directed spaces are exactly MD-spaces. As shown in \cite{Erne-2009, Yu-Kou-2015} (see also \cite{Luo-Xu-2017}), MD-spaces possess a lot of nice properties and have close relations with some other important spaces.

\begin{proposition}\label{lem-Erne-MD-space} (\cite[Proposition 7]{Erne-2009}) Let $P$ be a poset and $X$ a $T_0$-space. Then
\begin{enumerate}[\rm (1)]
\item  The weakest order-compatible MD-topology on $P$ is the Scott-Ern\'e topology $\sigma_2(P)$, the strongest one is the Alexandroff
topology $\alpha (P)$.
\item  If $X$ is locally hypercompact, then it is an MD-space.
\item The Scott space $\Sigma~\!\!P$ is an MD-space.
\item The monotone convergence MD-spaces are exactly the Scott spaces of dcpos.
\end{enumerate}
\end{proposition}

\begin{theorem}\label{theor-MD-cartesian-closed} (\cite[Theorem 3.10 and Theorem 4.8]{Yu-Kou-2015}) $\mathbf{MD}$ is cartesian closed and it is co-reflective in $\mathbf{Top}_0$.
\end{theorem}

\begin{definition}\label{def-md-topology}  (\cite[Definition 3.3]{Luo-Xu-2017})
Let $(X, \tau)$ be a $T_0$-space.
A subset $U\subseteq X$ is called \emph{monotone determined open} (shortly \emph{MD}-\emph{open}) if for any directed subset $D\subseteq X$, $\cl_{\tau}D \cap U\neq\emptyset$ implies $D\cap U\neq\emptyset$. All monotone determined open sets form a topology on $X$, denoted by $md(\tau)$.
\end{definition}


\begin{proposition}\label{prop-md-topology} (\cite[Lemma 3.4 and Theorem 3.5]{Luo-Xu-2017}) Let $(X, \tau)$ be a $T_0$-space. Then we have the following conclusions:
\begin{enumerate}[\rm (1)]
\item $\upsilon(X)\subseteq \tau\subseteq md(\tau)\subseteq \alpha(X)$.
\item $\leq_{\tau}=\leq_{md(\tau)}$ and hence $\mathcal{D}((X, \tau))=\mathcal{D}((X, md(\tau))$.
 \item $\cl_{\tau} D=\cl_{md(\tau)} D$ for any directed subset $D$ of $X$.
\item $\tau$ is an MD-topology iff $\tau=md(\tau)$.
\item $md(md(\tau))=md(\tau)$ and hence $md(\tau)$ is the weakest MD-topology containing $\tau$.
\end{enumerate}
\end{proposition}

\begin{proposition}\label{prop-md-upper-topology} (\cite[Theorem 3.8]{Luo-Xu-2017}) For a poset $P$, md$(\upsilon(P))=\sigma_2(P)$. Hence if $P$ is a dcpo, then md$(\upsilon(P))=\sigma (P)$.
\end{proposition}

A poset $P$ is said to be \emph{Noetherian} if it satisfies the \emph{ascending chain condition} ($\mathrm{ACC}$ for short): every ascending chain has a greatest member. It is easy to verify that $P$ is Noetherian if{}f every directed set of $P$ has a largest element. So every Noetherian poset is a dcpo.

\begin{proposition}\label{prop-Alexandroff-topology-sober} (\cite[Proposition 11]{Lawson-Xu-2024-2}).
	For a poset $P$, the following two conditions are equivalent:
	\begin{enumerate}[\rm (1)]
		\item $P$ is Noetherian.
        \item $P$ is a dcpo and $\alpha(P)=\sigma(P)$.
	\end{enumerate}
\end{proposition}

\section{$\mathcal{QS}$-convergence in $T_0$-spaces}

In this section, we study the $\mathcal{QS}$-convergence in $T_0$-spaces. Some characterizations are given for the $\mathcal{QS}$-convergence in a $T_0$-space being topological. It is shown that a $T_0$-space $(X,\tau)$ is locally hypercompact iff the $\mathcal{QS}$-convergence in $(X,\tau)$ coincides with the convergence in the topology $\tau$. Therefore, a poset $P$ is quasicontinuous iff $\mathcal{QS}$-convergence in the Scott space of $P$ is topological iff $\mathcal{QS}$-convergence coincides with convergence in the Scott topology $\sigma (P)$.

\begin{definition}\label{def-S-convergence}  (\cite[Definition 3.1]{Zhang-Bao-Xu-2022})
Let $(X, \tau)$ be a $T_0$-space, $x\in X$ and $\mathcal{N}$ be a net in $X$.
\begin{enumerate}[\rm (1)]

    \item We say that $x$ is an \emph{eventual lower bound} of $\mathcal{N}$ if $\mathcal{N}$ is eventually in ${\uparrow} x$.

    \item $\mathcal{N}$ is said to be $\mathcal{S}$-\emph{convergent} to a point $x\in X$,  written~$\mathcal{N}\stackrel{\mathcal{S}}\longrightarrow x$, if
there exists a directed subset $D\subseteq X$ of eventual lower bounds of $\mathcal{N}$ such that
$x\in \cl_{\tau} D$.

\end{enumerate}
\end{definition}

Let $\mathcal{S} = \{(\mathcal{N} , x) :\mathcal{N} \mbox{~is~a~net~in~} X,~x\in X \mbox{~and~} \mathcal{N}\stackrel{\mathcal{S}}\longrightarrow x \}$. Clearly, $\mathcal{O}(\mathcal{S})=\{U\subseteq X :$ whenever $(\mathcal{N} , x)\in \mathcal{S}$ and $x\in U$, then $\mathcal{N}$ is eventually in $U\}$  is a topology on $X$.

\begin{remark}\label{rem-S-convergence} (\cite[Remark 3.2]{Zhang-Bao-Xu-2022}) Let $P$ be a poset and $\mathcal{N}$ be a net in $P$.
\begin{enumerate}[\rm (1)]
   \item If $\mathcal{N}$ $\mathcal{S}$-converges to a point $x\in P$, then it $\mathcal{S}$-converges to every $y\in P$ with $y \leq x$.

   \item When $P$ is a $dcpo$ endowed with the Scott topology $\sigma(P)$, the preceding definition of $\mathcal{S}$-convergence
is equivalent to the standard one in \cite[Definition II-1.1]{GHKLMS-2003}.
\end{enumerate}
\end{remark}

\begin{proposition}\label{prop-O(S)=md} (\cite[Proposition 3.3]{Zhang-Bao-Xu-2022}) Let $(X, \tau)$ be a $T_0$-space. Then $\mathcal O(\mathcal S)=md(\tau)$.
\end{proposition}

\begin{definition}\label{def-quasi-eventual-lower-bound} (\cite[Definition 6]{Ruan-Xu-2019})
Let $(X, \tau)$ be a $T_0$-space and~$\mathcal{N}$~be a net in $X$. A nonempty finite set $F$ of $X$ is called a \emph{quasi-eventual lower bound} of $\mathcal{N}$ if $\mathcal{N}$ is eventually in ${\ua} F$.
\end{definition}

\begin{definition}\label{def-GS-convergence} (\cite[Definition 2.4]{Chen-Kou-2014})
Let $P$ be a dcpo, $x\in P$ and $\mathcal{N}$ be a net in $P$.
$\mathcal{N}$ is said to be $\mathcal{GS}$-\emph{convergent} to a point $x$, written $\mathcal{N}\stackrel{\mathcal{GS}}\longrightarrow x$, if there exists a family $\{F_d : d\in D\}$ of nonempty finite sets of $P$ satisfying the following three conditions:
\begin{enumerate}[\rm (1)]
\item $\{{\ua} F_d : d\in D\}$ is a filtered family, that is, for any $d_1, d_2\in D$, there is $d_3\in D$ such that $F_{d_3}\subseteq {\ua} F_{d_1}\cap {\ua} F_{d_2}$,

\item for each $d\in D$, $F_d$ is a quasi-eventual lower bound of $\mathcal{N}$, and

\item there is $a\in {\ua}x$ with ${\ua} a=\bigcap_{d\in D}{\ua} F_d$.
\end{enumerate}

Let $\mathcal{GS} = \{(\mathcal{N}, x) :\mathcal{N} \mbox{~is~a~net~in~} P,~x\in P \mbox{~and~}\mathcal{N}\stackrel{\mathcal{GS}}\longrightarrow x \}$ and $\mathcal{O}(\mathcal{GS})=\{U\subseteq P : $ whenever $(\mathcal{N}, x)\in \mathcal{GS}$ and $x\in U$, then $\mathcal{N}$ is eventually in $U\}$.
\end{definition}

\begin{proposition}\label{prop-GS-Scott-topology} (\cite[Proposition 2.7]{Chen-Kou-2014}) Let $P$ be a dcpo. Then $\mathcal{O}(\mathcal{GS})=\sigma(P)$.
\end{proposition}

\begin{definition}\label{def-QS-convergence} (\cite[Definition 3.2]{Chen-Kou-2024})
Let $(X, \tau)$ be a $T_0$-space, $x\in X$ and $\mathcal{N}$ be a net in $X$.
$\mathcal{N}$ is said to be \emph{quasi}-$\mathcal{S}$-\emph{convergent} (shortly $\mathcal{QS}$-\emph{convergent}) to a point $x$, written $\mathcal{N}\stackrel{\mathcal{QS}}\longrightarrow x$, if there exists a family $\{F_d : d\in D\}$ of nonempty finite sets of $X$ satisfying the following three conditions:
\begin{enumerate}[\rm (1)]
\item $\{{\ua} F_d : d\in D\}$ is a filtered family,

\item for each $d\in D$, $F_d$ is a  quasi-eventual lower bound of $\mathcal{N}$, and

\item ${\ua} x\in \cl_{P_S((X, \tau))}\{{\ua} F_d: d\in D\}$.
\end{enumerate}

Let $\mathcal{QS} = \{(\mathcal{N}, x) :\mathcal{N} \mbox{~is~a~net~in~} X,~x\in X \mbox{~and~}\mathcal{N}\stackrel{\mathcal{QS}}\longrightarrow x \}$. Then $\mathcal{O}(\mathcal{QS})=\{U\subseteq X : $ whenever $(\mathcal{N}, x)\in \mathcal{QS}$ and $x\in U$, then $\mathcal{N}$ is eventually in $U\}$  is a topology on $X$.
\end{definition}

\begin{remark}\label{rem-QS-convergence=Kou-L-convergence} Let $(X, \tau)$ be a $T_0$-space, $x\in X$ and $\mathcal{N}$ be a net in $X$. Suppose that $\{F_d : d\in D\}$ is a family of nonempty finite sets of $X$ satisfying conditions (1)-(3) of Definition \ref{def-QS-convergence}. It is easy to see that condition (3) is equivalent to the condition that the net $({\ua} F_d)_{d\in D}$ converges to ${\ua} x$ in $P_S(X)$ (note that $\{{\ua}F_d : d\in D\}\in \mathcal{D}(P_S((X, \tau)))$). So the $\mathcal{QS}$-convergence is equivalent to the $\mathcal S^*$-convergence introduced by Chen and Kou in \cite[Definition 3.2(ii)]{Chen-Kou-2024}. For more consistent notation with the $\mathcal S$-convergence, here we call such a convergence the $\mathcal{QS}$-convergence.
\end{remark}

\begin{remark}\label{rem-QS-convergence}
Let $(X, \tau)$ be a $T_0$-space, $x, y\in X$ and $\mathcal{N}$ be a net in $X$. Then we have the following conclusions:
\begin{enumerate}[\rm (1)]
       \item $\tau\subseteq \mathcal{O}(\mathcal{QS})$.

       Suppose that $(\mathcal{N}, x)\in \mathcal{QS}$ and $x\in U\in \tau $. Then there exists a family $\{F_d : d\in D\}$ of nonempty finite sets of $X$ such that $\{F_d : d\in D\}$ and $(\mathcal{N}, x)$ satisfy conditions (1)-(3) of Definition \ref{def-QS-convergence}. As $\ua x\in \Box U\in \mathcal O(P_S((X, \tau)))$ and $\ua x\in \cl_{P_S((X, \tau))}\{\ua F_d: d\in D\}$, there is $d\in D$ such that $\ua F_d\subseteq U$, whence by condition (2) of Definition \ref{def-QS-convergence}, $\mathcal{N}$ is eventually in $U$. Thus $U\in \mathcal{O}(\mathcal{QS})$.

       \item  If $\mathcal{N}\stackrel{\mathcal{S}}\longrightarrow x$, then $\mathcal{N}\stackrel{\mathcal{QS}}\longrightarrow x$. Hence $\mathcal{O}(\mathcal{QS})\subseteq \mathcal{O}(\mathcal{S})$.

       In fact, as $\mathcal{N}\stackrel{\mathcal{S}}\longrightarrow x$, there is a directed subset $D\subseteq X$ of eventual lower bounds of $\mathcal{N}$ such that
$x\in \cl_{\tau} D$. Then $\{\ua d : d\in D\}$ is a filtered family and for each $d\in D$, $\{d\}$ is a quasi-eventual lower bound of $\mathcal{N}$. By Lemma \ref{lem-closure-in-Smyth} we have $\ua x\in \Diamond \cl_{\tau} D=\cl_{P_S((X, \tau))}\{\ua d: d\in D\}$. Hence $\mathcal{N}\stackrel{\mathcal{QS}}\longrightarrow x$.

       \item If $\mathcal{N}\stackrel{\mathcal{QS}}\longrightarrow x$ and $y\leq x$, then $\mathcal{N}\stackrel{\mathcal{QS}}\longrightarrow y$.

\item For any $U\in \mathcal{O}(\mathcal{S})$, $U$ is an upper set.

            Suppose that $x\in U\in \mathcal{O}(\mathcal{S})$ and $x\leq y$. Consider the constant net $(y)_{i\in I}$ in $X$ with value $y$. Clearly, $(y)_{i\in I}\stackrel{\mathcal{S}}\longrightarrow y$, and hence $(y)_{i\in I}\stackrel{\mathcal{S}}\longrightarrow x$. As $x\in U\in \mathcal{O}(\mathcal{S})$, $(y)_{i\in I}$ is eventually in $U$, that is, $y\in U$. So $U$ is a upper set.

\item By (2) and (4), $U$ is an upper set for any $U\in \mathcal{O}(\mathcal{QS})$.

\end{enumerate}
\end{remark}

\begin{proposition}\label{prop-GS-is-QS} Let $P$ be a dcpo, $x\in P$ and $\mathcal{N}$ be a net in $P$.
If $\mathcal{N}\stackrel{\mathcal{GS}}\longrightarrow x$, then $\mathcal{N}\stackrel{\mathcal{QS}}\longrightarrow x$ in $(P. \sigma(P))$.
\end{proposition}
\begin{proof} As $\mathcal{N}\stackrel{\mathcal{GS}}\longrightarrow x$, there exists a family $\{F_d : d\in D\}$ of nonempty finite sets of $P$ satisfying the following three conditions:
\begin{enumerate}[\rm (1)]
\item $\{{\ua} F_d : d\in D\}$ is a filtered family, that is, for any $d_1, d_2\in D$, there is $d_3\in D$ such that $F_{d_3}\subseteq {\ua} F_{d_1}\cap {\ua} F_{d_2}$,

\item for each $d\in D$, $F_d$ is a quasi-eventual lower bound of $\mathcal{N}$, and

\item there is $a\in {\ua}x$ with ${\ua} a=\bigcap_{d\in D}{\ua} F_d$.
\end{enumerate}
Now we show that ${\ua} x\in \cl_{P_S((X, \tau))}\{{\ua} F_d: d\in D\}$. Let $U\in\sigma(P)$ with ${\ua}x\in \Box U$. Then $\bigcap_{d\in D}{\ua} F_d={\ua}a\subseteq {\ua}x\subseteq U$. By Corollary \ref{cor-Rudin-lemma}, there is $d\in D$ such that ${\ua}F_d\subseteq U$ or, equivalently, ${\ua}F_d\in \Box U$. Hence ${\ua} x\in \cl_{P_S((X, \tau))}\{{\ua} F_d: d\in D\}$. Therefore, $\mathcal{N}\stackrel{\mathcal{QS}}\longrightarrow x$ in $(P. \sigma(P))$.
\end{proof}

By the definition of $\mathcal{S}$-convergence and Remark \ref{rem-QS-convergence}(2), we have the following.

\begin{lemma}\label{lem-direct-set-S-convergence}
Let $(X, \tau)$ be a $T_0$-space. If $D\subseteq X$ is a directed set and $x\in cl_{\tau} D$, then the net $(d)_{d\in D}$ $\mathcal{S}$-converges to $x$. Hence it $\mathcal{QS}$-converges to $x$.
\end{lemma}

\begin{lemma}\label{lemma-filtered-family-QS-net}
Let $(X, \tau )$ be a $T_0$-space, $x\in X$ and $\{F_d : d\in D\}\subseteq X^{(<\omega)}$ such that $\{\ua F_d : d\in D\}$ is a filtered family and $\ua x\in \cl_{P_S((X, \tau))}\{\ua F_d: d\in D\}$. Then there is a net $\mathcal N$ in $\bigcup_{d\in D} F_d$ such that $(\mathcal N, x)\in \mathcal{QS}$.
\end{lemma}
\begin{proof}
Let $I=\{(F_d, n, e):d\in D, n\in\mathbb{N}$ and $e\in F_d\}$ and define an order on $I$ as follows: $(F_{d_1}, n_1, e_1)< (F_{d_2}, n_2, e_2)$ iff $\ua F_{d_2}$ is a proper subset of $\ua F_{d_1}$ or $\ua F_{d_2}=\ua F_{d_1}$ and $n_2>n_1$. Then $I$ is a directed set. Define $x_{(F_{d}, n, e)}=e$ for all $(F_{d}, n, e)\in I$. It is easy to see that $F_d$ is a quasi-eventual lower bound of  $(x_i)_{i\in I}$  for each $d\in D$. As $\{\ua F_d : d\in D\}$ is a filtered family and $\ua x\in \cl_{P_S((X, \tau))}\{\ua F_d: d\in D\}$, $(x_i)_{i\in I}$ $\mathcal{QS}$-converges to $x$.
\end{proof}

\begin{proposition}\label{prop-O(QS)=O(S)} (\cite[Lemma 3.5]{Chen-Kou-2024})
Let $(X, \tau)$ be a $T_0$-space. Then $\mathcal{O}(\mathcal{QS})=\mathcal{O}(\mathcal{S})=md(\tau)$.
\end{proposition}

By Proposition \ref{lem-Erne-MD-space}(3), Proposition \ref{prop-md-topology}(4) and Proposition \ref{prop-O(QS)=O(S)}, we get the following.

\begin{corollary}\label{cor-Scott-topology-O(QS)=O(S)=sigma}
For a poset $P$ and its Scott space $(P, \sigma(P))$, we have $\mathcal{O}(\mathcal{QS})=\mathcal{O}(\mathcal{S})=\sigma(P)$.
\end{corollary}

\begin{proposition}\label{prop-O(QS)-charac}
Let $(X, \tau)$ be a $T_0$-space and $U\subseteq X$. Then the following two conditions are equivalent:
\begin{enumerate}[\rm (1)]
    \item   $U\in \mathcal{O}(\mathcal{QS})$.

      \item  $U=\ua U$ and for any filtered family $\mathcal{F}\subseteq \mathbf{Fin}~\!\!X$, $\Box U\cap \cl_{P_S((X, \tau))}\mathcal F\neq\emptyset $ implies $\Box U\cap \mathcal F\neq\emptyset$.
\end{enumerate}
\end{proposition}
\begin{proof} (1) $\Rightarrow$ (2): Let $U\in \mathcal{O}(\mathcal{QS})$. Then by Proposition \ref{prop-O(S)=md} and Proposition \ref{prop-O(QS)=O(S)}, we have $U\in md(\tau)$.
  Suppose that $\mathcal{F}\subseteq \mathbf{Fin}~\!\!X$ is a filtered family satisfying $\Box U\cap \cl_{P_S((X, \tau))}\mathcal F \neq\emptyset$. If $\ua F\nsubseteq U$ for all $\ua F\in \mathcal{F}$, then by Lemma \ref{lem-filtred-family} $\{\ua (F\setminus U) : \ua F\in \mathcal{F}\}\subseteq \mathbf{Fin}~\!\!X$ is a filtered family.
By Lemma \ref{lem-Rudin-lemma}, there exists a directed subset $E\subseteq \mathop{\bigcup}\limits_{\ua F \in \mathcal{F}} (F\setminus U)$ such that $E\cap (F\setminus U)\neq\emptyset$ for all $\ua F\in \mathcal{F}$. Then $\cl_{P_S((X, \tau))}\mathcal{F}\subseteq \Diamond \cl_{\tau}E$, whence $\Box U\cap \Diamond \cl_{\tau}E\neq\emptyset$. Therefore, $U\cap \cl_{\tau} E\neq\emptyset$, and hence $U\cap E\neq\emptyset$ by $U\in md(\tau)$, which contradicts $E\subseteq \mathop{\bigcup}\limits_{\ua F \in \mathcal{F}} (F\setminus U)\subseteq X\setminus U$. So there exists $\ua F\in \mathcal{F}$ such that $\ua F\subseteq U$. Thus $\Box U\cap \mathcal{F}\neq\emptyset$.

(2) $\Rightarrow$ (1):  Suppose that  $U=\ua U$ and $\mathcal{N}$ $\mathcal{QS}$-converges to $x\in U$. Then there exists a family $\{F_d : d\in D\}$ of nonempty finite sets of $X$ such that $\{F_d : d\in D\}$ and $(\mathcal N, x)$ satisfy conditions (1)-(3) of Definition \ref{def-QS-convergence}. Hence $\ua x\in \Box U\cap\cl_{P_S((X, \tau))}\{\ua F_d: d\in D\}$. By (2), we have $\Box U\cap \{\ua F_d: d\in D\}\neq\emptyset$, whence there is $d^\prime\in D$ with $\ua F_{d^\prime}\subseteq U$. As $\{F_d : d\in D\}$ satisfies condition (2) of Definition \ref{def-QS-convergence}, $\mathcal{N}$ is eventually in $\ua F_{d^\prime}$ and hence it is eventually in $U$. So $U\in \mathcal{O}(\mathcal{QS})$.
\end{proof}

By Proposition \ref{lem-Erne-MD-space}(2), Proposition \ref{prop-md-topology}(4) and Proposition \ref{prop-O(QS)=O(S)}, we deduce the following.

\begin{corollary}\label{cor-LHC-O(GS)=tau}
For a locally hypercompact $T_0$-space $(X, \tau)$, $\mathcal{O}(\mathcal{QS})=\mathcal{O}(\mathcal{S})=\tau$.
\end{corollary}

The following example shows that for a $T_0$-space $(X,\tau)$, $\mathcal{O}(\mathcal{QS})=\tau$ does not hold in general.

\begin{example}\label{exam-O(QS)-is-not-tau}
Let $L =\{a_n : n\in \mathbb{N}^+\} \cup \{\bot, \top \})$ and define an order on $L$ as
follows (see Figure 1):

\begin{enumerate}[\rm (1)]
\item~$\bot< a_n < \top$ for all $n\in \mathbb{N}^+$, and

\item~$a_n$ and $a_m$ are are incomparable for any $n, m\in \mathbf{N}^+$ with $n\neq m$.
\end{enumerate}

\begin{figure}[htbp]
  \centering
\begin{tikzpicture}[scale=1,
    dot/.style={circle,fill,inner sep=2pt},
]

\node[dot,label=above:$\top$] (T) at (0.2,4) {};

\node[dot,label=below left:$a_1$] (a1) at (-3,2) {};
\node[dot,label=below:$a_2$]     (a2) at (-1.8,2) {};
\node[dot,label=below:$a_3$]     (a3) at (-0.6,2) {};
\node[font=\small]               (dots) at (0.3,2) {$\dots$};
\node[dot,label=below right:$a_n$](an) at (1.2,2) {};
\node[dot,label=below right:$a_{n+1}$](a$_{n+1}$) at (2.4,2) {};
\node[dot,label=below right:$a_{n+2}$](a$_{n+2}$) at (3.6,2) {};
\node[font=\small]               (dots) at (4.4,2) {$\dots$};

\node[dot,label=below:$\bot$] (bot) at (0.2,0) {};

\draw (T) -- (a1);
\draw (T) -- (a2);
\draw (T) -- (a3);
\draw (T) -- (an);
\draw (T) -- (a$_{n+1}$);
\draw (T) -- (a$_{n+2}$);

\draw (a1) -- (bot);
\draw (a2) -- (bot);
\draw (a3) -- (bot);
\draw (an) -- (bot);
\draw (a$_{n+1}$) -- (bot);
\draw (a$_{n+2}$) -- (bot);
\end{tikzpicture}

  \caption{The complete lattice $L$ in Example \ref{exam-O(QS)-is-not-tau}}
  \label{fig:poset}
\end{figure}
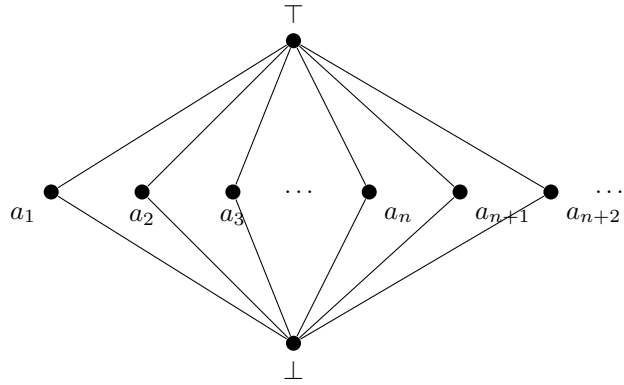

Clearly, $L$ is a countable complete lattice and Noetherian. Consider the space $(L, \upsilon (L))$. Then by Proposition \ref{prop-md-upper-topology}, Proposition \ref{prop-Alexandroff-topology-sober} and Proposition \ref{prop-O(QS)=O(S)}, we have $\mathcal{O}(\mathcal{QS})=\mathcal{O}(\mathcal{S})=md(\upsilon (L))=\sigma (L)=\alpha(L)$. As $\{\top\}\in \sigma(L)$ and $\{\top\}\not\in \upsilon(L)$, we get $\mathcal{O}(\mathcal{QS})\neq \upsilon(L)$.
\end{example}

\begin{proposition}\label{prop-clusure-of-direct-set-in-O(S)} (\cite[Proposition 3.6]{Zhang-Bao-Xu-2022})
Let $(X, \tau)$ be a $T_0$-space. Then we have the following conclusions:
\begin{enumerate}[\rm (1)]
  \item   $(X, \mathcal{O}(\mathcal{S}))$ is a $T_0$-space.

\item   The specialization orders of spaces $(X, \tau )$ and $(X, \mathcal{O}(\mathcal{S}))$ coincide.

  \item   For any directed subset $D\subseteq X$, $\cl_{\tau}D=\cl_{\mathcal{O}(\mathcal{S})}D$.
\end{enumerate}
\end{proposition}

By Remark  \ref{rem-QS-convergence}(4),  Proposition \ref{prop-O(QS)=O(S)} and Proposition \ref{prop-clusure-of-direct-set-in-O(S)}, we get the following.

\begin{corollary}\label{cor-upsilon-tau-QS-alpha}
Let $(X, \tau)$ be a $T_0$-space. Then we have the following conclusions:
\begin{enumerate}[\rm (1)]

  \item   $\upsilon(X)\subseteq \tau\subseteq md(\tau)=\mathcal{O}(\mathcal{S})= \mathcal{O}(\mathcal{QS})\subseteq \alpha(X)$.

  \item     $(X, \mathcal{O}(\mathcal{QS}))$ is a $T_0$-space and the specialization orders of spaces $(X, \tau )$ and $(X, \mathcal{O}(\mathcal{QS}))$ coincide.

  \item    For any directed subset $D\subseteq X$, $\cl_{\tau}D=\cl_{\mathcal{O}(\mathcal{S})}D=\cl_{\mathcal{O}(\mathcal{QS})}D$.
\end{enumerate}
\end{corollary}

Now we give a characterization of $T_0$-spaces for the $\mathcal{QS}$-convergence being topological. To this end, we first recall the following notion.

\begin{definition} (\cite[Definition 3.1]{Feng-Kou-2017})
Let $(X, \tau )$ be a $T_0$-space, $A, B\subseteq X$. We say that $A$ \emph{approximates} $B$, in symbols $A\ll_w B$,
if for any directed set $D$ in $X$, $B\cap \cl_{\tau}D\neq\emptyset$  implies ${\ua} A\cap D\neq\emptyset$.
We write $A\ll_{w} x$ for $A\ll_{w} \{x\}$ and $y\ll_{w} B$ for $\{y\}\ll_{w} B$. For $x\in X$ and $F\in X^{(<\omega)}$, define ${\dda_{w}}x=\{G\in X^{(<\omega)}: G\ll_{w} x\}$ and ${\dua_{w}}F=\{z\in X: F\ll_{w} z\}$.
\end{definition}

\begin{lemma}\label{lem-characterization-way-below-c-relation-2}
Let $(X, \tau)$ be a $T_0$-space, $F\in X^{(<\omega)}$ and $y\in X$. Then the following two conditions are equivalent:
\begin{enumerate}[\rm (1)]

\item $F\ll_w y$.

\item For any net $\mathcal{N}$ in $X$ with $(\mathcal{N}, y)\in \mathcal{QS}$, $\mathcal{N}$ is eventually in ${\ua} F$.
\end{enumerate}
\end{lemma}
\begin{proof} (1) $\Rightarrow$ (2): Suppose that $F\ll_w y$ and $\mathcal{N}$ is a net in $X$ with $(\mathcal{N}, y)\in \mathcal{QS}$. Then there exists a family $\{F_d : d\in D\}\subseteq X^{(<\omega)}$ such that $\{F_d : d\in D\}$ and $(\mathcal N, y)$ satisfy conditions (1)-(3) of Definition \ref{def-QS-convergence}. We claim that there is $d\in D$ such that ${\ua} F_d\subseteq {\ua} F$. Assume, on the contrary, that ${\uparrow} F_d\nsubseteq {\ua} F$ for all $d\in D$. As $\{{\ua} F_{d} : d\in D\}$ is a filtered family, it follows from Lemma \ref{lem-filtred-family} that $\{{\ua} (F_d\setminus {\ua} F) : d\in D\}$ is a filtered family. By Lemma \ref{lem-Rudin-lemma}, there is a directed subset $E\subseteq \mathop{\bigcup}\limits_{d \in D } (F_d\setminus {\ua} F)$ such that $E\cap (F_d\setminus {\ua} F)\neq\emptyset$ for all $d\in D$.
  Then ${\ua} y\in \cl_{P_S((X, \tau))}\{{\ua} F_d: d\in D\}\subseteq \Diamond \cl_{\tau}E$, whence $ y\in \cl_{\tau}E$. As $F\ll_w y$, we have $E\cap {\ua} F\neq \emptyset$, which contradicts $E\subseteq \mathop{\bigcup}\limits_{d \in D } (F_d\setminus {\ua} F)\subseteq X\setminus {\ua} F$. So there is a $d\in D$ such that ${\uparrow} F_d\subseteq{\ua} F$. Since $F_d$ is a quasi-eventual lower bound of $\mathcal{N}$, we get that $\mathcal{N}$ is eventually in ${\ua} F$.

(2) $\Rightarrow$ (1): Let $D$ be a  directed subset  of $X$ with $y\in \cl_{\tau}D$. Then $\{\{d\} : d\in D\}$ and $((d)_{d\in D}, y)$ satisfy conditions (1)-(3) of Definition \ref{def-QS-convergence}. Hence $((d)_{d\in D}, y)\in \mathcal{QS}$. By (2), $(d)_{d\in D}$ is eventually in ${\ua} F$, whence $D\cap {\ua} F\neq\emptyset$. Thus $F\ll_w y$.
\end{proof}

\begin{proposition}\label{prop-QS-topological} If $\mathcal{QS}$-convergence in a $T_0$-space $X$ is topological, then ${\dua_{w}}F\in \mathcal{O}(\mathcal{QS})$ for any $F\in X^{(<\omega)}$.
\end{proposition}
\begin{proof} Let $F\in X^{(<\omega)}$. We show ${\dua_{w}}F\in md(\mathcal O(X))$. Assume that $D\in \mathcal D(X)$ and $\dua_w F\cap
\cl~\!\!  D\neq\emptyset$. Then there is $u\in \cl D$ such that $F\ll_w u$. By Theorem \ref{theor-QS-topological-charact} and \cite[Proposition 3.8]{Feng-Kou-2017} (note that Proposition 3.8 of \cite{Feng-Kou-2017} holds for any $T_0$-space satisfying condition (2) of \cite[Definition 3.5]{Feng-Kou-2017}), there is $H\in X^{(<\omega)}$ such that $F\ll_w H\ll_w u$. Then by $H\ll_w u$ and $u\in \cl D$, we have $\emptyset\neq\ua H\cap D\subseteq \dua_w F\cap D$. So  ${\dua_{w}}F\in md(\mathcal O(X))$, and hence ${\dua_{w}}F\in \mathcal O(\mathcal{QS})$ by Proposition \ref{prop-O(QS)=O(S)}.
\end{proof}

\begin{theorem}\label{theor-QS-topological-charact} (\cite[Theorem 3.6 and Corollary 3.9]{Chen-Kou-2024})
Let $X$ be a $T_0$-space. Then the following two conditions are equivalent:
\begin{enumerate}[\rm (1)]
   \item   $\mathcal{QS}$-convergence in $X$ is topological, that is, for all $x\in X$ and all nets $(x_i)_{i\in I}$ in $X$,
\begin{center}
$(x_i)_{i\in I}$  $\stackrel{QS}\longrightarrow x$ in $X$ if and only if $(x_i)_{i\in I}$ converges to $x$ with respect to the topology $\mathcal{O}(\mathcal{QS})$.
\end{center}

  \item   For any $x \in X$, the following two conditions hold:
          \begin{enumerate}[\rm (i)]

                \item    $\{{\ua} F : F\in X^{(<\omega)}, F\ll_{w}x\}$ is a filtered family.

                 \item    ${\ua} x\in \cl_{P_S(X)}\{{\ua} F : F\in X^{(<\omega)}, F\ll_{w}x\}$.

%
%
%
%
              \end{enumerate}

\end{enumerate}

\indent If $X$ is an MD-space, these conditions are equivalent to the
following one:

\vspace{0.2cm}

\noindent $\mathrm{(3)}$ $X$ is locally hypercompact.
\end{theorem}

\begin{theorem}\label{theor-LCH-QS-topological}
For a $T_0$-space $(X, \tau)$, the following three conditions are equivalent:
\begin{enumerate}[\rm (1)]
   \item  $(X, \tau)$ is locally hypercompact.

\item $\mathcal{QS}$-convergence is topological and $\mathcal O(\mathcal{QS})=\tau$.

   \item $\mathcal{QS}$-convergence coincides with convergence in the topology $\tau$, that is, for any $x\in X$ and any net $\mathcal{N}$ in $X$,
\begin{center}
$\mathcal{N}\stackrel{\mathcal{QS}}\longrightarrow x$ iff $\mathcal{N}$ converges to $x$ with respect to the topology $\tau$.
\end{center}
\end{enumerate}
\end{theorem}
\begin{proof}
(1) $\Rightarrow$ (2): By Corollary \ref{cor-LHC-O(GS)=tau} and Theorem \ref{theor-QS-topological-charact}.

(2) $\Rightarrow$ (3): Trivial.

(3) $\Rightarrow$ (1): As $\mathcal{QS}$-convergence coincides with convergence in the topology $\tau$, we have
$\mathcal O(\mathcal{QS})=\tau$, whence $md(\tau)=\mathcal O(\mathcal{QS})=\tau$ by Proposition \ref{prop-O(QS)=O(S)}. Therefore, $(X, \tau)$ is locally hypercompact by Theorem \ref{theor-QS-topological-charact}.
\end{proof}

From Corollary \ref{cor-Scott-topology-O(QS)=O(S)=sigma} and Theorem \ref{theor-LCH-QS-topological} we deduce the following.

\begin{corollary}\label{cor-quasicontinuous-poset-QS-convergennce}
For a poset $P$, the following conditions are equivalent:
\begin{enumerate}[\rm (1)]
  \item  $P$ is a quasicontinuous poset.

  \item  $\mathcal{QS}$-convergence in $\Sigma~\!\!P$ is topological.

  \item  $\mathcal{QS}$-convergence in $\Sigma~\!\!P$ coincides with convergence in the Scott topology $\sigma (P)$.
\end{enumerate}
\end{corollary}

   At the end of this section, we investigate conditions under which the mapping $f : (X, \mathcal{O}(\mathcal{S})) \rightarrow (Y, \mathcal{O}(\mathcal{S}))$ is continuous for a given mapping $f : (X, \tau)  \rightarrow (Y, \eta)$ between two $T_0$-spaces.

   In what follows, for a $T_0$-space $(X, \tau)$,
in order to emphasize the topology $\tau$, we will use $\mathcal{O}_{\tau}(\mathcal{S})$ (resp., $\mathcal{O}_{\tau}(\mathcal{QS})$) to denote  $\mathcal{O}(\mathcal{S})$ (resp., $\mathcal{O}(\mathcal{QS})$).

\begin{definition}\label{def-S-continuous}  (\cite[Definition 3.9]{Zhang-Bao-Xu-2022})
A mapping $f: (X, \tau)\rightarrow (Y, \eta)$ between $T_0$-spaces is called $\mathcal{S}$-\emph{continuous} if $f^{-1}(U)\in\mathcal{O}_{\tau}(\mathcal{S})$ for all $U\in \mathcal{O}_{\eta}(\mathcal{S})$, that is, $f : (X, \mathcal{O}_\tau(\mathcal{S})) \rightarrow (Y, \mathcal{O}_\eta(\mathcal{S}))$ is continuous.
\end{definition}

Similarly, $f: (X, \tau)\rightarrow (Y, \eta)$ is said to be $\mathcal{QS}$-\emph{continuous} if $f : (X, \mathcal{O}_\tau(\mathcal{QS})) \rightarrow (Y, \mathcal{O}_\eta(\mathcal{QS}))$ is continuous.

\begin{lemma}\label{lem-S-continuous} (\cite[Lemma 3.10]{Zhang-Bao-Xu-2022})
For a mapping  $f: (X, \tau)\rightarrow (Y, \eta)$ between two $T_0$-spaces, the following three conditions are
equivalent:
\begin{enumerate}[\rm (1)]
  \item   $f$ is $\mathcal{S}$-continuous.

  \item   $f(\cl_{\tau}D)\subseteq \cl_{\eta}f(D)$ for all directed subsets $D \subseteq X$.

  \item   For any net $\mathcal{N}$ in $X$ and $x\in  X$, $\mathcal{N}\stackrel{\mathcal{S}}\longrightarrow x$ in $(X, \tau)$ implies $f(\mathcal{N})\stackrel{\mathcal{S}}\longrightarrow f(x)$ in $(Y, \eta)$.

\end{enumerate}
\end{lemma}

\begin{corollary}
For a mapping  $f: (X, \tau)\rightarrow (Y, \eta)$ between two $T_0$-spaces, the following conditions are
equivalent:
\begin{enumerate}[\rm (1)]

\item $f$ is $\mathcal{S}$-continuous.

  \item $f : (X, md(\tau)) \rightarrow (Y, md(\eta))$ is continuous.

  \item $f$ is $\mathcal{QS}$-continuous.

 \item  $f(\cl_{\tau}D)\subseteq \cl_{\eta}f(D)$ for all directed subsets $D \subseteq X$.

  \item  For any net $\mathcal{N}$ in $X$ and $x\in X$, $\mathcal{N}\stackrel{\mathcal{QS}}\longrightarrow x$ in $(X, \tau)$ implies $f(\mathcal{N})\stackrel{\mathcal{QS}}\longrightarrow f(x)$ in $(Y, \eta)$.
\end{enumerate}
\end{corollary}
\begin{proof} (1) $\Leftrightarrow$ (2) $\Leftrightarrow$ (3) $\Leftrightarrow$ (4): By Proposition \ref{prop-O(QS)=O(S)} and Lemma \ref{lem-S-continuous}.

(1) $\Rightarrow$ (5): Firstly, we show that $f$ is order-preserving. Suppose that $x\leq y$ in $X$. Then by (1) and Lemma \ref{lem-S-continuous}, we have $f(x)\in f(\cl_{\tau}\{y\})\subseteq \cl_{\eta}\{f(y)\}$. So $f(x)\leq f(y)$.
Let $\mathcal{N}$ be a net in $X$ and $x\in X$ with $\mathcal{N}\stackrel{\mathcal{QS}}\longrightarrow x$ in $(X, \tau)$. Then there exists a family  $\{F_d : d\in D\}\subseteq X^{(<\omega)}$ of quasi-eventual lower bounds of $\mathcal{N}$ such that $\{{\ua} F_d : d\in D\}$ is a filtered family and
${\ua} x\in \cl_{P_S((X, \tau))}\{{\ua} F_d: d\in D\}$. For any $d\in D$, as $f$ is order-preserving, we have that ${\ua} f({\ua} F_d)={\ua} f(F_d)$. Hence $\{f(F_d) : d\in D\}\subseteq Y^{(<\omega)}$ is a family of quasi-eventual lower bounds of $f(\mathcal{N})$ and $\{{\ua} f(F_d) : d\in D\}$ is a filtered family. Now we show that ${\ua} f(x)\in \cl_{P_S((Y, \eta))}\{{\ua} f(F_d): d\in D\}$. Let $V\in \eta$ with $f(x)\in V$. Then $x\in f^{-1}(V)$. We claim that there is $d\in D$ such that $F_d\subseteq f^{-1}(V)$. Assume, on the contrary, that $F_d\nsubseteq f^{-1}(V)$ for all $d\in D$.
Since $f$ is order-preserving, $f^{-1}(V)$ is an upper set. By Lemma \ref{lem-filtred-family}, $\{{\ua} (F_d\setminus f^{-1}(V)) : d\in D\}\subseteq \mathbf{Fin}~\!\!X$ is a filtered family. Then by Lemma \ref{lem-Rudin-lemma} there is a directed subset $E$ of $(X, \tau)$ with $E\subseteq \bigcup_{d\in D}(F_d\setminus f^{-1}(V))$ such that $E\cap (F_d\setminus f^{-1}(V))\neq\emptyset$ for all $d\in D$. Hence $\{{\ua} F_d : d\in D\}\subseteq \Diamond \cl_{\tau}E$. So ${\ua} x\in \cl_{P_S((X, \tau))}\{{\ua} F_d : d\in D\}\subseteq \Diamond \cl_{\tau}E$, and hence $x\in \cl_{\tau}E$. By (1) and Lemma \ref{lem-S-continuous}, we get $f(x)\in f(\cl_{\tau}E)\subseteq \cl_{\eta}f(E)$, whence $V\cap\cl_{\eta}f(E) \neq\emptyset$. As $E\subseteq \bigcup_{d\in D}(F_d\setminus f^{-1}(V))$, we have $f(E)\cap V=\emptyset$, and hence $V\cap \cl_{\eta}f(E)=\emptyset$ by $V\in \eta$, a contradiction. So there is $d^\prime\in D$ such that $F_{d^\prime}\subseteq f^{-1}(V)$, and consequently, $\{{\ua} f(F_d): d\in D\}\cap \Box V\neq\emptyset$. Thus ${\ua} f(x)\in \cl_{P_S((Y, \eta))}\{{\ua} f(F_d): d\in D\}$, and this completes the proof that $f(\mathcal{N})\stackrel{\mathcal{QS}}\longrightarrow f(x)$ in $(Y, \eta)$.

(5) $\Rightarrow$ (1): Let $W\in \mathcal{O}_{\eta}(\mathcal{QS})$. Now we show that $f^{-1}(W)\in \mathcal{O}_{\tau}(\mathcal{QS})$. Suppose that $D$ is a directed subset with $f^{-1}(W)\cap\cl_{\tau}D\neq\emptyset$. Choose a point $x\in f^{-1}(W)\cap\cl_{\tau}D$. Then by Lemma \ref{lem-direct-set-S-convergence}, $(d)_{d\in D}$ $\mathcal{QS}$-converges to $x$. By (5), $(f(d))_{d\in D}$ $\mathcal{QS}$-converges to $f(x)$. Hence by $f(x)\in W\in \mathcal{O}_{\eta}(\mathcal{QS})$ and Proposition \ref{prop-O(QS)=O(S)}, we have $(f(d))_{d\in D}$ is eventually in $W$. So $D\cap f^{-1}(W)\neq\emptyset$. By Proposition \ref{prop-O(QS)=O(S)}, we have $f^{-1}(W)\in md(\tau)=\mathcal{O}_{\tau}(\mathcal{QS})$. Thus $f$ is $\mathcal{QS}$-continuous.
\end{proof}

Let $(X, \tau)$, $(Y, \eta)$ be two $T_0$-spaces. A pair $(f, g)$ of mappings $f: (X, \tau) \rightarrow (Y, \eta)$ and $g: (Y, \eta) \rightarrow (X, \tau)$ is said to be an \emph{adjunction}, provided that both $f$ and $g$ are order-preserving, and for any $(x, y)\in X\times Y$, $f(x)\geq y$ iff $x\geq g(y)$. In an adjunction $(f, g)$, the mapping $f$ is called the \emph{upper adjoint} and $g$ the \emph{lower adjoint}.

\begin{proposition}
Let $X, Y$ be two $T_0$-spaces and $f: (X, \tau)\rightarrow (Y, \eta)$ be the upper adjoint of $g: (Y, \eta)\rightarrow (X, \tau)$.
Consider the following three conditions:
\begin{enumerate}[\rm (1)]
   \item    $f$ is $\mathcal{S}$-continuous.

   \item   For any $V\in \mathcal{O}_{\eta}(\mathcal{S})$, ${\ua} g(V)\in \mathcal{O}_{\tau}(\mathcal{S})$.

   \item   For any $F\in Y^{(<\omega)}$, $g(\ii_{\mathcal{O}_{\eta}(\mathcal{S})}{\ua} F)\subseteq $ $\ii_{\mathcal{O}_{\tau}(\mathcal{S})}{\ua} g(F)$.

\end{enumerate}
Then \emph{(1)} $\Leftrightarrow$ \emph{(2)} $\Rightarrow$ \emph{(3)}, and all three conditions are equivalent if $(Y, \eta)$ is locally hypercompact.
\end{proposition}
\begin{proof} (1) $\Leftrightarrow$ (2): By \cite[Proposition 3.13]{Zhang-Bao-Xu-2022}.

(2) $\Rightarrow$ (3): Let $F\in Y^{(<\omega)}$ and $y\in Y$ with $y\in \ii_{\mathcal{O}_{\eta}(\mathcal{S})}{\ua} F$. Then there is $V\in \mathcal{O}_{\eta}(\mathcal{S}) $ such that $y\in V \subseteq{\ua} F$, whence $g(y)\in g(V)\subseteq g({\ua} F)\subseteq {\ua} g(F)$. By (2), we have $g(y)\in {\ua} g(V)\in \mathcal{O}_{\tau}(\mathcal{S})$ and ${\ua} g(V)\subseteq {\ua} g(F)$. Hence  $g(y)\in \ii_{\mathcal{O}_{\tau}(\mathcal{S})}{\ua} g(F)$. Thus $g(\ii_{\mathcal{O}_{\eta}(\mathcal{S})}{\ua} F)\subseteq $ $\ii_{\mathcal{O}_{\tau}(\mathcal{S})}{\ua} g(F)$.

(3) $\Rightarrow$ (1): Assume that $(Y, \eta)$ is locally hypercompact and $ V\in \mathcal{O}_{\eta}(\mathcal{S})$. Then by Corollary \ref{cor-LHC-O(GS)=tau} $\mathcal{O}_{\eta}(\mathcal{S})=\eta$, whence $V\in \eta$. Now we show that $f^{-1}(V)\in \mathcal{O}_{\tau}(\mathcal{S})$. Suppose that $D$ is a directed set in $(X, \tau)$ with $f^{-1}(V)\cap \cl_{\tau}D\neq\emptyset$. Then there is $x\in f^{-1}(V)\cap \cl_{\tau}D$ and hence $f(x)\in V\in \eta$. As $(Y, \eta)$ is locally hypercompact, there exists $G\in Y^{(<\omega)}$ such that $f(x)\in \ii_{\eta}{\ua} G\subseteq {\ua} G\subseteq V$. By (3) we get $g(f(x))\in g(\ii_{\eta}{\ua} G)=g(\ii_{\mathcal{O}_{\eta}(\mathcal{S})}{\ua} G)\subseteq \ii_{\mathcal{O}_{\tau}(\mathcal{S})}{\ua} g(G)$. Since $g(f(x))\leq x$, we get $x\in \ii_{\mathcal{O}_{\tau}(\mathcal{S})}{\ua} g(G)$. By Proposition \ref{prop-md-topology}(3) and Corollary \ref{cor-LHC-O(GS)=tau}, we have $x\in \cl_{\tau}D=\cl_{md(\tau)}D=\cl_{\mathcal{O}_{\tau}(\mathcal{S})}D$. Hence $\emptyset \neq D\cap \ii_{\mathcal{O}_{\tau}(\mathcal{S})}{\ua} g(G)\subseteq {\ua} g(G)\cap D$. As $f$ is the upper adjoint of $g$, we have ${\ua} G\cap f(D)\neq\emptyset$, and consequently,  $V\cap f(D)\neq\emptyset$, or equivalently, $f^{-1}(V)\cap D\neq\emptyset$. Therefore, $f^{-1}(V)\in md(\tau)$, whence $f^{-1}(V)\in \mathcal{O}_{\tau}(\mathcal{S})$ by Proposition \ref{prop-O(S)=md}. Thus $f$ is $\mathcal{S}$-continuous.
\end{proof}

\begin{definition} Let $(X, \tau)$ and  $(Y, \eta)$ be two $T_0$-spaces. A mapping $f: (X, \tau)\rightarrow (Y, \eta)$ is called $\mathcal F$-\emph{open} if for any $F\in X^{(<\omega)}$,  $f(\ii_{\mathcal{O}_{\tau}(\mathcal{S})}{\ua} F)\subseteq $ $\ii_{\mathcal{O}_{\eta}(\mathcal{S})}{\ua} f(F)$.
\end{definition}

Now we introduce the following two categories:

\vskip 1mm

(1) $\textbf{SC}_G$ has as objects locally hypercompact $T_0$-spaces and as morphisms $\mathcal{S}$-continuous mappings that have a lower adjoint.

\vskip 1mm

(2) $\textbf{SC}_D$ has as objects locally hypercompact $T_0$-spaces and as morphisms mappings which are $\mathcal F$-open and have an upper adjoint.

\vskip 1mm

Define functors $F :\textbf{SC}_D\rightarrow \textbf{SC}_G$ and $G :\textbf{SC}_G\rightarrow \textbf{SC}_D$, where $F(X) = X$ and $G(X) = X$ for any locally hypercompact $T_0$-space $X$. If a mapping $g : Y \rightarrow X$ has an upper adjoint $f$, then this upper adjoint is uniquely determined by $g$, and we denote it by $f = F(g) : X\rightarrow Y$.
 And if a mapping $f : X \rightarrow Y$ has a lower adjoint $g : Y \rightarrow X$, then this lower adjoint $g$ is uniquely determined by $f$, and we denote it by $g = G(f) : Y \rightarrow X$.

With a similar argument to that of the proof of Theorem IV-1.3 in \cite{GHKLMS-2003}, we get the following result.

\begin{theorem}
The categories $\textbf{SC}_D$ and $\textbf{SC}_G$ are dual under the functors $F$ and $G$ given through the
Galois connection of functions.
\end{theorem}

\section{Quasi-liminf convergence in $T_0$-spaces}

In this section, we study the quasi-liminf convergence in $T_0$-spaces. A new kind of $T_0$-spaces, called \emph{WLH}-spaces, are introduced and investigated. We prove that every locally hypercompact $T_0$-space is a \emph{WLH}-space, and a $T_0$-space $(X, \tau)$ is a \emph{WLH}-space iff the quasi-liminf convergence in $(X, \tau)$ is topological. Therefore, the quasi-liminf convergence in $(X, \tau)$ is topological if  $(X, \tau)$ is locally hypercompact, and for a quasicontinuous poset $P$, the quasi-liminf convergence is topological and agrees with convergence in the Lawson topology $\lambda(P)$. Using the quasi-liminf convergence and approximate relation $\ll_c$, we give  several characterizations of $C$-spaces and continuous posets.

\begin{definition} (\cite[Definition 4.1]{Zhang-Bao-Xu-2022}) Let $X$ be a $T_0$-space, $x\in  X$ and $\mathcal{N}$ be a net in $X$.
\begin{enumerate}[\rm (1)]
    \item $x$ is said to be the \emph{liminf} of $\mathcal{N}$, written $\mathcal{N}\stackrel{l}\longrightarrow x$,  if  $\mathcal{N}\stackrel{\mathcal{S}}\longrightarrow x$  and $x\in {\ua} y$  for any eventual lower bound
 $y$ of $\mathcal{N}$.

    \item $\mathcal{N}$  is said to be \emph{liminf convergent} to $x$, written $\mathcal{N}\stackrel{L}\longrightarrow x$, if $\mathcal{N}_f\stackrel{l}\longrightarrow x$ for all subnets $\mathcal{N}_f$ of $\mathcal{N}$.

        \item Let $\mathcal{L}$ = $\{(\mathcal{N} , x) : \mathcal{N}\stackrel{L}\longrightarrow x\}$.
The topology $\mathcal{O}(\mathcal{L}) = \{U \subseteq  X:$ whenever $(\mathcal{N}, x) \in \mathcal{L}$ and
$x \in U$, then $\mathcal{N}$ is eventually in $U\}$ is called the \emph{liminf topology} on $X$.

\end{enumerate}
\end{definition}

\begin{proposition}\label{prop-liminf-convergence-charac} (\cite[Proposition 4.3]{Zhang-Bao-Xu-2022})
Let $X$ be a $T_0$-space, $x \in X$ and $\mathcal{N}$ a net in $X$. Then the following two conditions are
equivalent:
\begin{enumerate}[\rm (1)]

    \item $(\mathcal N, x)\in \mathcal{L}$.

     \item There exists a directed subset $D\subseteq X$ such that

\begin{enumerate}[\rm (i)]
\item ${\da} x=\cl~\!\! D$;
\item for each $d\in D$, $d$ is a eventual lower bound of $\mathcal N$;
\item for any $z\in X$, if $\mathcal N$ is usually in ${\ua} z$, then $x\in {\ua} z$.
\end{enumerate}
\end{enumerate}
\end{proposition}

\begin{lemma}\label{lem-Lawson-tau-O(L)} (\cite[Lemma 4.2]{Zhang-Bao-Xu-2022})
Let $(X, \tau )$ be  a $T_0$-space. Then we have the following two conclusions:

\begin{enumerate}[\rm (1)]
  \item   $\omega(X)\vee \tau\subseteq \mathcal{O}(\mathcal{L})$.

  \item   If $(X, \tau )$ is a $C$-space, then $\mathcal{O}(\mathcal{L})=\omega(X)\vee \tau$.
\end{enumerate}
\end{lemma}

\begin{definition} Let $X$ be a $T_0$-space, $x\in  X$ and $\mathcal{N}$ be a net in $X$.
\begin{enumerate}[\rm (1)]
    \item $x$ is said to be the \emph{quasi-liminf} of $\mathcal{N}$, written $\mathcal{N}\stackrel{ql}\longrightarrow x$, if  $\mathcal{N}\stackrel{\mathcal{QS}}\longrightarrow x$  and $x\in {\ua} F$  for any quasi-eventual lower bound
 $F$ of $\mathcal{N}$.

    \item $\mathcal{N}$  is said to be \emph{quasi-liminf convergent} to $x$, written $\mathcal{N}\stackrel{QL}\longrightarrow x$, if $\mathcal{N}_f\stackrel{ql}\longrightarrow x$ for all subnets $\mathcal{N}_f$ of $\mathcal{N}$.

        \item Let $\mathcal{QL}$ = $\{(\mathcal{N} , x) : \mathcal{N}\stackrel{QL}\longrightarrow x\}$.
The topology $\mathcal{O}(\mathcal{QL}) = \{U \subseteq  X:$ whenever $(\mathcal{N}, x) \in \mathcal{QL}$ and
$x \in U$, then $\mathcal{N}$ is eventually in $U\}$ is called the \emph{quasi-liminf topology} on $X$.
\end{enumerate}
\end{definition}

\begin{proposition}\label{prop-quasi-liminf-convergence-charac}

Let $X$ be a $T_0$-space, $x \in X$ and $\mathcal{N}$ a net in $X$. Then the following four conditions are
equivalent:
\begin{enumerate}[\rm (1)]

    \item $(\mathcal N, x)\in \mathcal{QL}$.

     \item There exists a family $\{F_d : d\in D\}\subseteq X^{(<\omega)}$ such that

\begin{enumerate}[\rm (i)]
\item $\{{\ua} F_d : d\in D\}$ is a filtered family;
\item for each $d\in D$, $F_d$ is a quasi-eventual lower bound of $\mathcal N$;
\item for any $F\in X^{(<\omega)}$, if $\mathcal N$ is usually in ${\ua} F$, then $x\in {\ua} F$; and
\item $\cl_{P_S(X)}\{{\ua} F_d : d\in D\}=\cl_{P_S(X)}\{{\ua} x\}$.
\end{enumerate}
\item $(\mathcal N, x)\in \mathcal{QS}$ and for any $F\in X^{(<\omega)}$, if $\mathcal N$ is usually in ${\ua} F$, then $x\in {\ua} F$;
\item $(\mathcal N, x)\in \mathcal{QS}$ and for any $z\in X$, if $\mathcal N$ is usually in ${\ua} z$, then $x\in {\ua} z$.

\end{enumerate}
\end{proposition}

\begin{proof} (1) $\Rightarrow$ (2): As $\mathcal N$ is a subnet of itself, we have $\mathcal N\stackrel{ql}\longrightarrow x$. Then there exists a family $\{F_d : d\in D\}\subseteq X^{(<\omega)}$ of quasi-eventual lower bounds of $\mathcal N$ such that $\{{\ua} F_d : d\in D\}$ is a filtered family and ${\ua} x\in\cl_{P_S(X)}\{{\ua} F_d : d\in D\}$ (equivalently, $\cl_{P_S(X)}\{{\ua} x\}\subseteq\cl_{P_S(X)}\{{\ua} F_d : d\in D\}$). For any $F\in X^{(<\omega)}$, if $\mathcal N$ is usually in ${\ua} F$, then there is a subset $\mathcal N_f$ of $\mathcal N$ such that
$\mathcal N_f$ is in ${\ua} F$. By $\mathcal N\stackrel{QL}\longrightarrow x$, we get $\mathcal N_f\stackrel{ql}\longrightarrow x$ and hence $x\in {\ua} F$. In particular, for any $d\in D$, we have $x\in {\ua} F_d$ or, equivalently, ${\ua} F_d\in \cl_{P_S(X)}\{{\ua} x\}$. So $\cl_{P_S(X)}\{{\ua} F_d : d\in D\}\subseteq \cl_{P_S(X)}\{{\ua} x\}$ and hence $\cl_{P_S(X)}\{{\ua} x\}=\cl_{P_S(X)}\{{\ua} F_d : d\in D\}$.

(2) $\Rightarrow$ (3): Trivial.

(3) $\Rightarrow$ (1): As $(\mathcal N, x)\in \mathcal{QS}$, there exists a family $\{F_d : d\in D\}\subseteq X^{(<\omega)}$ of quasi-eventual lower bounds of $\mathcal N$ such that $\{{\ua} F_d : d\in D\}$ is a filtered family and ${\ua} x\in\cl_{P_S(X)}\{{\ua} F_d : d\in D\}$. Let $\mathcal N_f$ be a subnet of $\mathcal N$. Then $\mathcal{N}_f\stackrel{\mathcal{QS}}\longrightarrow x$. For $F\in X^{(<\omega)}$, if $F$ is a quasi-eventual lower bound of $\mathcal N_f$, then $\mathcal N$ is usually in ${\ua} F$. Hence $x\in {\ua} F$ by (3). Therefore,  $\mathcal{N}_f\stackrel{ql}\longrightarrow x$. Thus $(\mathcal N , x)\in \mathcal{QL}$.

 (3) $\Rightarrow$ (4): Trivial.

 (4) $\Rightarrow$ (3): Let $(\mathcal N, x)\in \mathcal{QS}$ and  $F=\{u_1, u_2, ..., u_n\}\in X^{(<\omega)}$. Suppose that $\mathcal N$ is usually in ${\ua} F$. If for any $1\leq i\leq n$, $\mathcal N$ is not usually in ${\ua} u_i$. Then for each $i\in \{1, 2, ..., n\}$, $\mathcal N$ is eventually in $X\setminus {\ua} u_i$, and hence $\mathcal N$ is eventually in $\bigcap\limits_{i=1}^n (X\setminus {\ua} u_i)=X\setminus \bigcup\limits_{i=1}^n{\ua} u_i=X\setminus {\ua} F$, which contradicts the assumption that $\mathcal N$ is usually in ${\ua} F$. So there is $i\in \{1, 2, ..., n\}$ such that $\mathcal N$ is usually in ${\ua} u_i$. Thus by (4) we have $x\in {\ua} u_i\subseteq{\ua} F$.
\end{proof}

\begin{remark}\label{rem-QL-convergence=Kou-L-convergence}  Let $(X, \tau)$ be a $T_o$-space and $P$ be a dcpo.
\begin{enumerate}[\rm (a)]
\item By the equivalence of conditions (1) and (4) of Proposition \ref{prop-quasi-liminf-convergence-charac}, the $\mathcal{QL}$-convergence in $(X, \tau)$ is equivalent to the $\mathcal{L}^*$-convergence $(X, \tau)$ introduced by Chen and Kou in
    \cite[Definition 4.3]{Chen-Kou-2024}. For more consistent notation with the $\mathcal L$-convergence, here we call such a convergence the $\mathcal{QL}$-convergence.
\item For the Scott space $(P, \sigma (P))$, the quasi-limiinf convergence was first introduced by Zhou and Li \cite{Zhou-Li-2013}, which was called the $\mathscr{S}^*_2$-\emph{convergence} in \cite{Zhou-Li-2013}.
\end{enumerate}
\end{remark}

\begin{lemma}\label{lem-direct-set-QL}
Let $(X, \tau )$ be a $T_0$-space, $x\in X$ and $D$ be a directed subset in $X$. If $\cl_{\tau}D={\da} x$, then $((d)_{d\in D}, x)\in \mathcal{QL}$.
\end{lemma}
\begin{proof} First, we show that $\cl_{P_S((X, \tau))}\{{\ua} d : d\in D\}=\cl_{P_S((X, \tau))}\{{\ua} x\}$. For any $U\in \mathcal O(X)$, we have that
$$\begin{array}{lll}
{\ua} x\in \Box U&\Leftrightarrow&{\da} x\cap U=U\cap \cl_{\tau} \{x\}\neq\emptyset\\
&\Leftrightarrow&U\cap\cl_{\tau} D\neq\emptyset ~(\mbox{by~}\cl_{\tau}D={\da} x)\\
&\Leftrightarrow&U\cap D\neq\emptyset\\
&\Leftrightarrow&\{{\ua} d : d\in D\}\cap \Box U\neq\emptyset.
\end{array}$$

\noindent Hence $\cl_{P_S((X, \tau))}\{{\ua} d : d\in D\}=\cl_{P_S((X, \tau))}\{{\ua} x\}$.
Now consider the family $\{\{d\} : d\in D\}$. Clearly, $\{{\ua} d={\ua} \{d\} : d\in D\}$ is a filtered family, and for each $d^\prime\in D$, $\{d^\prime\}$ is a quasi-eventual lower bound of $(d)_{d\in D}$, that is, $(d)_{d\in D}$ is eventually in ${\ua} d^\prime$. Assume that $F\in X^{(<\omega)}$ for which $(d)_{d\in D}$ is usually in ${\ua} F$. Then there is $d^*$ such that $d^*\in {\ua} F$, whence by Remark \ref{rem-Smyth-specialization-order} ${\ua} F\in \cl_{P_S((X, \tau))}\{{\ua} d^*\}\subseteq \cl_{P_S((X, \tau))}\{{\ua} d : d\in D\}=\cl_{P_S((X, \tau))}\{{\ua} x\}$. By Remark \ref{rem-Smyth-specialization-order} again, we have ${\ua} x\subseteq {\ua} F$. So $\{\{d\} : d\in D\}$ and $((d)_{d\in D}, x)$ satisfy conditions (i)-(iv) of Proposition \ref{prop-quasi-liminf-convergence-charac}(2). Therefore, $((d)_{d\in D}, x)\in \mathcal{QL}$ by Proposition \ref{prop-quasi-liminf-convergence-charac}.
\end{proof}

More generally, we have the following (comparing it with Lemma \ref{lemma-filtered-family-QS-net}).

\begin{lemma}\label{lem-filtered-QL}
Let $(X, \tau )$ be a $T_0$-space, $x\in X$ and $\{F_d : d\in D\}\subseteq X^{(<\omega)}$ such that $\{{\ua} F_d : d\in D\}$ is a filtered family and $\cl_{P_S((X, \tau))}\{{\ua} F_d: d\in D\}=\cl_{P_S((X, \tau))}\{{\ua} x\}$. Then there is a net $\mathcal N$ in $\bigcup_{d\in D} F_d$ such that $(\mathcal N, x)\in \mathcal{QL}$.
\end{lemma}
\begin{proof} For any $d^\prime\in D$, we have ${\ua} F_{d^\prime}\in \cl_{P_S((X, \tau))}\{{\ua} F_d: d\in D\}=\cl_{P_S((X, \tau))}\{{\ua} x\}$, whence by Remark \ref{rem-Smyth-specialization-order} ${\ua} x\subseteq {\ua} F_{d^\prime}$ or, equivalently, $F_{d^\prime}\cap {\da} x\neq\emptyset$. Therefore, $F_d\cap {\da} x$ is a non-empty finite set for any $d\in D$. By Lemma \ref{lem-filtred-family}  $\{{\ua} (F_d\cap {\da} x) : d\in D\}$ is a filtered family.

Let $I=\{(F_d\cap {\da} x, n, e): d\in D, n\in \mathbb{N}$ and $e\in  F_d\cap {\da} x\}$ and define an order on $I$ as follows: $(F_{d_1}\cap {\da} x, n_1, e_1)< (F_{d_2}\cap {\da} x, n_2, e_2)$ iff ${\ua} (F_{d_2}\cap {\da} x)$ is a proper subset of ${\ua} (F_{d_1}\cap {\da} x)$ or ${\ua} (F_{d_2}\cap {\da} x)={\ua} (F_{d_1}\cap {\da} x)$ and $n_1<n_2$. Then $I$ is a directed set. Define $x_{(F_d\cap {\da} x, n, e)}=e$ for any $(F_d\cap {\da} x, n,  e)\in I$.
 Then we get a net $\mathcal N=(x_i)_{i\in I}$. Now we show that $\{F_d  : d\in D\}$ and ($\mathcal N, x)$ satisfy conditions (ii) and (iii) of Proposition \ref{prop-quasi-liminf-convergence-charac}(2). Let $d\in D$. Choose an $e\in F_d\cap {\da} x$ and $m\in \mathbb{N}$. Then $(F_d\cap {\da} x, m, e)\in I$ and for any $(F_{d^\prime}\cap {\da} x, n, e^\prime)> (F_d\cap {\da} x, m, e)$ in $I$, we have $e^\prime=x_{(F_{d^\prime}\cap {\da} x, n, e^\prime)}\in F_{d^\prime}\cap {\da} x\subseteq {\ua} (F_d\cap {\da} x)\subseteq{\ua} F_d$. Hence $F_d$ is a quasi-eventual lower bound of $(x_i)_{i\in I}$. Assume that $F\in X^{(<\omega)}$ for which $\mathcal N$ is usually in ${\ua} F$. Then there is $(F_{d*}\cap {\da} x, n^*, e^*)\in I$ such that $e^*=x_{(F_{d^*}\cap {\da} x, n^*, e^*)}\in {\ua} F$. As $e^*\in F_{d*}\cap {\da} x$, we have $e^*\leq x$ and hence $x\in {\ua} F$. Therefore, $\{F_d : d\in D\}$ and $(\mathcal N, x)$ satisfy conditions (i)-(iv) of Proposition \ref{prop-quasi-liminf-convergence-charac}(2). It follows from Proposition \ref{prop-quasi-liminf-convergence-charac} that $(\mathcal N, x)\in \mathcal{QL}$.
\end{proof}

\begin{remark}\label{rem-simpler-proof} Using Lemma \ref{lem-direct-set-QL} we can provide a relatively simpler proof of Lemma \ref{lem-filtered-QL}. Indeed, by Lemma \ref{lem-Rudin-lemma}, there exists a directed set $E\subseteq \bigcup_{d\in D}F_d\cap {\da} x$ such that $E\cap F_d\cap {\da} x\neq\emptyset$ for all $d\in D$. Hence $\{{\ua} (F_d\cap {\da} x) : d\in D\}\subseteq\Diamond \cl_{\tau}E$, and consequently, ${\ua} x\in\cl_{P_S((X, \tau))}\{{\ua} x\}=\cl_{P_S((X, \tau))}\{{\ua} F_d: d\in D\}\subseteq \cl_{P_S((X, \tau))}\{{\ua} (F_d\cap {\da} x): d\in D\}\subseteq\Diamond \cl_{\tau}E$. Then $x\in \cl_{\tau}E$, whence ${\da} x\subseteq\cl_\tau E$. As $E\subseteq \bigcup_{d\in D} F_d\cap {\da} x\subseteq {\da} x$, we have $\cl_{\tau}E\subseteq {\da} x$. Thus $\cl_\tau E={\da} x$. By Lemma \ref{lem-direct-set-QL},
we have $((e)_{e\in E}, x)\in \mathcal{QL}$. Let $\mathcal N=(e)_{e\in E}$. Then $\mathcal N\subseteq\bigcup_{d\in D} F_d$ and $(\mathcal N, x)\in \mathcal{QL}$.
\end{remark}

\begin{proposition}\label{prop-L-convergence-QL-convergence}
Let $X$ be a $T_0$-space. Then $\mathcal{L}\subseteq\mathcal{QL}$ and hence $\mathcal{O}(\mathcal{QL})\subseteq \mathcal{O}(\mathcal{L})$.
\end{proposition}

\begin{proof}
Let $((x_j)_{j\in J}, x)\in \mathcal{L}$. Now we show that  $((x_j)_{j\in J}, x)\in \mathcal{QL}$. As $((x_j)_{j\in J}, x)\in \mathcal{L}$, there exists a directed subset $D\subseteq X$ such that $D$ and $((x_j)_{j\in J}, x)$ satisfy conditions (i)-(iii) of Proposition \ref{prop-liminf-convergence-charac}(2). Then $\{\{d\} : d\in D\}$ and $((x_j)_{j\in J}, x)$ satisfy conditions (i), (ii) and (iv) of Proposition \ref{prop-quasi-liminf-convergence-charac}(2). Let $F=\{f_1, f_2, \ldots, f_n\}\in X^{(<\omega)}$. Suppose that $(x_j)_{j\in J}$ is usually in ${\ua} F$. We claim that there is a $1\leq i_0\leq n $ such that $\mathcal N$ is usually in ${\ua} f_{i_0}$. Assume, on the contrary, that for any $1\leq i\leq n$, $(x_j)_{j\in J}$ is not usually in ${\ua} f_{i}$. Then for each $1\leq i\leq n$, $(x_j)_{j\in J}$ is eventually in  $X\setminus {\ua} f_{i}$. Hence $(x_j)_{j\in J}$ is eventually in $\bigcap\limits_{i=1}^n(X\setminus {\ua} f_{i})=X\setminus \bigcup\limits_{i=1}^n{\ua} f_i=X\setminus {\ua} F$, which contradicts the assumption that $(x_j)_{j\in J}$ is usually in ${\ua} F$. Therefore, $(x_j)_{j\in J}$ is usually in ${\ua} f_{i_0}$ for some $i_0\in \{1, 2, \ldots,  n\}$. As  $((x_j)_{j\in J}, x)\in \mathcal{L}$, by Proposition \ref{prop-liminf-convergence-charac} (2)(iii) we have $x\in {\ua} f_{i_0}$, and hence $x\in {\ua} F$. Therefore, $\{\{d\} : d\in D\}$ and $((x_j)_{j\in J}, x)$ satisfy conditions (i)-(iv) of Proposition \ref{prop-quasi-liminf-convergence-charac}(2). Thus  $((x_j)_{j\in J}, x)\in \mathcal{QL}$. So $\mathcal{L}\subseteq\mathcal{QL}$, and hence $\mathcal{O}(\mathcal{QL})\subseteq \mathcal{O}(\mathcal{L})$.
\end{proof}

From Proposition \ref{prop-O(QS)=O(S)} we know that for any $T_0$-space $X$, $\mathcal{O}(\mathcal{QS})=\mathcal{O}(\mathcal{S})$. But we do not know whether $\mathcal{O}(\mathcal{QL})=\mathcal{O}(\mathcal{L})$ always holds. So we pose the following question.

\begin{question}
  For a $T_0$-space $X$, does  $\mathcal{O}(\mathcal{QL})= \mathcal{O}(\mathcal{L})$ hold?
\end{question}

\begin{lemma}\label{lem-Lawson-tau-O(QL)} \emph{(\cite[Proposition 4.4 and Theorem 4.5]{Chen-Kou-2024} )}
Let $(X, \tau )$ be  a $T_0$-space. Then we have the following two conclusions:

\begin{enumerate}[\rm (1)]
  \item   $\omega(X)\vee \tau\subseteq \mathcal{O}(\mathcal{QL})$.

  \item  If $(X, \tau )$ is locally hypercompact, then $\mathcal{QL}$-convergence is topological and $\mathcal{O}(\mathcal{QL})=\omega(X)\vee \tau$.

\end{enumerate}
\end{lemma}

\begin{proposition}\label{prop-Lawson-tau-O(QL)-O(L)}
Let $(X, \tau )$ be a $T_0$-space. Then we have the following conclusions:

\begin{enumerate}[\rm (1)]

  \item $\omega(X)\vee \tau\subseteq \omega(X)\vee md(\tau)=\omega(X)\vee \mathcal O(\mathcal{QS})\subseteq\mathcal{O}(\mathcal{QL})\subseteq \mathcal O(\mathcal{L})$.

  \item   If $(X, \tau )$ is locally hypercompact, then $\omega(X)\vee \tau=\omega(X)\vee md(\tau)=\mathcal{O}(\mathcal{QL})\subseteq  \mathcal{O}(\mathcal{L})$.

  \item   If $(X, \tau )$ is a $C$-space, then $\omega(X)\vee \tau=\omega(X)\vee md(\tau)=\mathcal{O}(\mathcal{QL})=\mathcal{O}(\mathcal{L})$.
\end{enumerate}
\end{proposition}
\begin{proof} (1): Clearly, $\mathcal{QL}\subseteq\mathcal{QS}$, whence $\mathcal O(\mathcal{QS}) \subseteq\mathcal{O}(\mathcal{QL})$. Therefore, by Corollary \ref{cor-upsilon-tau-QS-alpha}(1), Proposition \ref{prop-L-convergence-QL-convergence} and Lemma \ref{lem-Lawson-tau-O(QL)}(1), we have $\omega(X)\vee \tau\subseteq \omega(X)\vee md(\tau)=\omega(X)\vee \mathcal O(\mathcal{QS})\subseteq\mathcal{O}(\mathcal{QL})\subseteq \mathcal O(\mathcal{L})$.

(2): As $(X, \tau )$ is locally hypercompact, by Proposition \ref{lem-Erne-MD-space}(2) and Proposition \ref{prop-md-topology}(4) we have $\tau=md(\tau)$. Therefore, by (1) and Lemma \ref{lem-Lawson-tau-O(QL)}(2) we get $\omega(X)\vee \tau=\omega(X)\vee md(\tau)=\mathcal{O}(\mathcal{QL})\subseteq  \mathcal{O}(\mathcal{L})$.

(3): By (2) and Lemma \ref{lem-Lawson-tau-O(L)}(2).
\end{proof}

\begin{definition}\label{def-ll-c-relation} (\cite[Definition 4.5]{Zhang-Bao-Xu-2022})
 Let $(X, \tau )$ be a $T_0$-space, $x, y\in X$. We say that $x$ \emph{approximates} $y$, in symbols $x\ll_c y$,
if for any directed set $D$ in $X$, $\cl_{\tau}D={\downarrow} y$  implies $x\in{\da} D$.  Let ${\Downarrow_c} x = \{ y \in X : y \ll_c x \} $ and ${\Uparrow_c} x = \{ y \in X : x \ll_c y \} $.
\end{definition}

More generally, we introduce the following.

\begin{definition} Let $(X, \tau )$ be a $T_0$-space, $A, B\subseteq X$. We say that $A$ \emph{approximates} $B$, in symbols $A\ll_c B$,
if for any directed set $D$ in $X$, $\cl_{\tau}D={\downarrow} b$ for some $b\in B$ implies $D\cap {\uparrow} A\neq\emptyset$.
We write $A\ll_{c} x$ for $A\ll_{c} \{x\}$ and $y\ll_{c} B$ for $\{y\}\ll_{c} B$. For $x\in X$ and $F\in X^{(<\omega)}$, we write $w_c(x)=\{{\ua} F : F\in X^{(<\omega)}, F\ll_{c}x\}$ and ${\dua_{c}}F=\{x\in X : F\ll_{c} x\}$.
\end{definition}

In what follows, for a pair of subsets $A, B$ of a $T_0$-space $X$, we use the symbol $A\leq B$ to denote ${\ua} B\subseteq {\ua} A$ for convenience. We can easily verify the following facts.

\begin{remark}\label{rem-way-below-c}
Let $X$ be a $T_0$-space and $H$, $K$, $M\subseteq X$. Then we have the following conclusions:
\begin{enumerate}[\rm (1)]
  \item  $G\ll_{c} H$ $\Leftrightarrow$ $G\ll_{c} h$ for all $h\in H$.

  \item  $G\ll_{c} H$ $\Rightarrow$ $G\leq H$ (that is, ${\ua} H\subseteq {\ua} G$).

  \item  $G\leq H\ll_{c} K$ $\Rightarrow$ $G\ll_{c} K$.

  \item If $x\in \ii_{\mathcal{O}(\mathcal{QL})}{\ua} F$, then $F\ll_{c} x$ by Lemma \ref{lem-direct-set-QL}.

  \item  $G\ll_{c} x$ $\Rightarrow$  $(G\cap {\da} x)\ll_{c} x$.
\end{enumerate}
\end{remark}

The following example illustrates that for a locally hypercompact $T_0$-space $(X, \tau)$, $\dua_{c}F$ is not always an upper set for $F\in X^{(<\omega)}$.

\begin{example}\label{exam-locally-hypercompact-not-upper}
Let $P =\{\top\}\cup \{a_n : n \in \mathbb{N}\}\cup \{b_n : n \in \mathbb{N}\}$ (see Figure 2) with the order generated by
\begin{enumerate}[\rm (1)]
 \item   $a_n\leq a_m$ for all $m, n\in \mathbb{N}$ with $n\leq m$;

 \item   $b_n\leq b_m$ for all $m, n\in \mathbb{N}$ with $n\leq m$; and

 \item   $a_n< \top$ and $b_n< \top$ for all $n\in \mathbb{N}$.
\end{enumerate}

\begin{figure}[h]
    \centering
    \includegraphics[width=12cm, height=4cm, keepaspectratio]{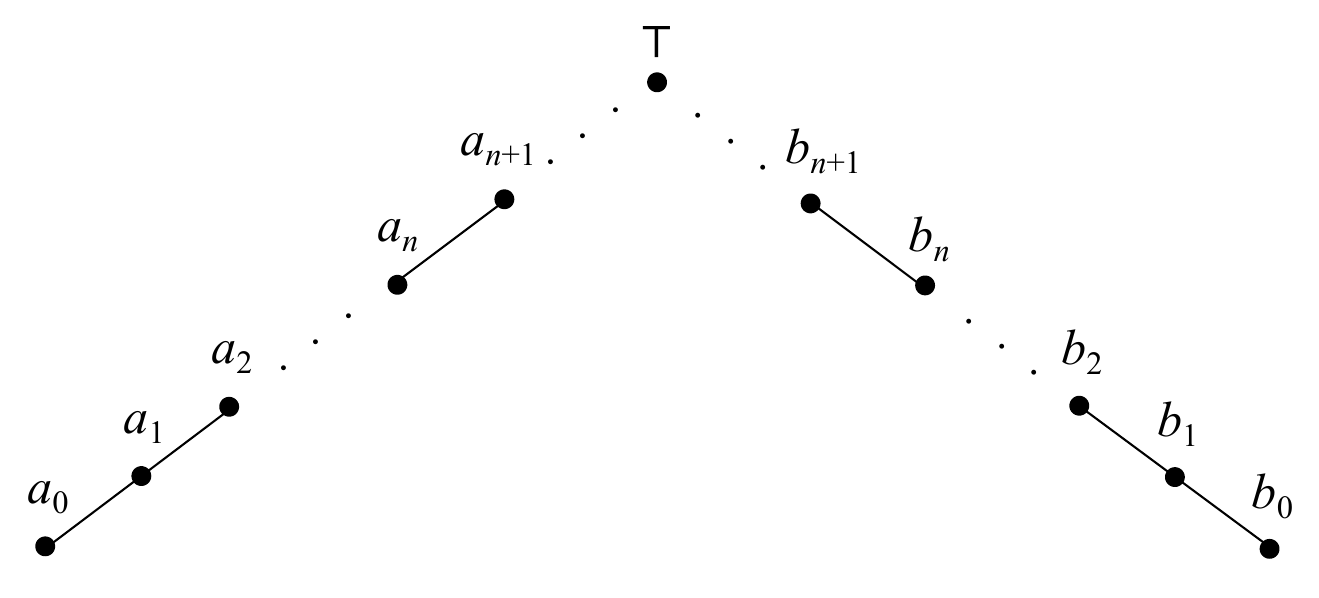}
    \caption{The countable quansicontinuous domain $P$ in Example \ref{exam-locally-hypercompact-not-upper}}
\end{figure}

\noindent Clearly, $P$ is a dcpo and $\mathrm{Id}(P)=\{{\da} x : x\in P\}$. It is easy to verify that $P$ is a quasicontinuous domain, and consequently, by Proposition \ref{prop-quasicontinuous-local-hypercompact} the Scott space $\Sigma~\!\!P$ is locally hypercompact. Let $F=\{a_0\}$. As $\cl_{\sigma(P)}\{b_n: n\in \mathbb{N}\}={\da} \top$ and $a_0\notin \{b_n: n\in \mathbb{N}\}$, we have $\top\notin \dua_{c}F$.
By $\mathrm{Id}(P)=\{{\da} x : x\in P\}$ and ${\da}\top=P=\cl_{\sigma(P)}\{b_n : n\in\mathbb{N}$, we get ${\dua_{c}}F=\{a_n : n\in \mathbb{N}\}$ and $\top\not\in {\dua_{c}}F$. Hence ${\dua_{c}}F$ is not an upper set.
\end{example}

\begin{lemma}\label{lem-characterization-way-below-c-relation-1}
Let $(X, \tau)$ be a $T_0$-space, $F\in X^{(<\omega)}$ and $y\in X$. Then the following two conditions are equivalent:
\begin{enumerate}[\rm (1)]

\item $F\ll_c y$.

 \item For any net $\mathcal{N}$ in $X$ with $(\mathcal{N}, y)\in \mathcal{QL}$, $\mathcal{N}$ is eventually in ${\ua} F$.
\end{enumerate}
\end{lemma}
\begin{proof} (1) $\Rightarrow$ (2): Suppose that $F\ll_c y$ and $\mathcal{N}$ is a net in $X$ with $(\mathcal{N}, y)\in \mathcal{QL}$. Then there exists a family $\{F_d : d\in D\}\subseteq X^{(<\omega)}$ such that $\{F_d : d\in D\}$ and $(\mathcal N, y)$ satisfy conditions (i)-(iv) of Proposition \ref{prop-quasi-liminf-convergence-charac}(2).

 {\bf Claim 1:} For any $d\in D$, $F_d\cap {\da} y$ is nonempty.

 As $\{F_d : d\in D\}$ and $(\mathcal N, y)$ satisfy condition (iv) of Proposition \ref{prop-quasi-liminf-convergence-charac}(2), we have ${\ua} F_d\in \cl_{P_S((X, \tau))}\{{\ua} y\}$, whence ${\uparrow} y\subseteq {\ua} F_d$, or equivalently, $F_d\cap {\da} y\neq\emptyset$.

  {\bf Claim 2:} For any $d\in D$, ${\ua} (F_d\cap {\da} y)={\ua} ({\ua} F_d\cap {\da} y)$.

  Clearly,  ${\ua} (F_d\cap {\da} y)\subseteq{\ua} ({\ua} F_d\cap {\da} y)$. On the other hand, let $u\in {\ua} ({\ua} F_d\cap {\da} y)$. Then there is $v\in {\ua} F_d\cap {\da} y$ such that $v\leq u$, and consequently, $w\leq v\leq y$ for some $w\in F_d$. So $w\in F_d\cap {\da} y$ and $w\leq u$. Hence $u\in {\ua} (F_d\cap {\da} y)$. Therefore, ${\ua} ({\ua} F_d\cap {\da} y)\subseteq{\ua} (F_d^\prime\cap {\da} y)$. Thus ${\ua} (F_d\cap {\da} y)={\ua} ({\ua} F_d\cap {\da} y)$.

  {\bf Claim 3:}  For any $d\in D$, $F_d\cap {\da} y$  is a quasi-eventual lower bound of $\mathcal{N}$.

Let $\mathcal{N}=(x_i)_{i\in I}$ and $d\in D$.  Assume, on the contrary, that $F_d\cap {\da} y$ is not a quasi-eventual lower
bound of $\mathcal{N}$. As $\mathcal N$ is eventually in ${\ua} F_d$, there is $i_0\in \mathcal{N}$ such that $x_i\in {\ua} F_d$ for all $i\geq i_0$. For any $j\in I$, since $\mathcal N$ is usually in $X\setminus {\ua}(F_d\cap {\da} y)$, there is $i^{\prime}\in I$ with $i_0\leq i^{\prime}$ and $j\leq i^{\prime}$ such that $x_{i^{\prime}}\in X\setminus {\ua} (F_d\cap {\da} y)$. Then $x_i^{\prime}\in {\ua} F_d$, and hence there is $u\in F_d$ with $u\leq x_{i^{\prime}}$. So $u\in F_d\setminus {\da} y$ by $x_{i^{\prime}}\notin {\ua} (F_d\cap {\da} y)$. Therefore, $\mathcal N$ is usually in ${\ua} (F_d\setminus {\da} y)$. Since $\{F_{d} : d\in D\}\subseteq X^{(<\omega)}$ and $(\mathcal N, y)$ satisfy condition (iii) of Proposition \ref{prop-quasi-liminf-convergence-charac}(2), we have $y\in {\ua} (F_d\setminus {\da} y)\subseteq X\setminus {\da} y$, a contradiction. Thus $F_d\cap {\da} y$ is a quasi-eventual lower bound of $\mathcal{N}$.

 {\bf Claim 4:} $\cl_{P_S((X, \tau))}\{{\ua} (F_d\cap {\da} y): d\in D\}=\cl_{P_S((X, \tau))}\{{\ua} y\}$.

By Remark \ref{rem-Smyth-specialization-order} and  condition (iv) of Proposition \ref{prop-quasi-liminf-convergence-charac}(2), we have that  ${\ua} y\in\cl_{P_S((X, \tau))}\{{\ua} F_d: d\in D\}\subseteq\cl_{P_S((X, \tau))}\{{\ua} (F_d\cap {\da} y): d\in D\}$ and ${\ua} y \subseteq {\ua} (F_d\cap {\da} y)$ for any $d\in D$. Hence $\cl_{P_S((X, \tau))}\{{\ua} (F_d\cap {\da} y): d\in D\}=\cl_{P_S((X, \tau))}\{{\ua} y\}$.

   {\bf Claim 5:}  There is $d^\ast\in D$ such that $F_{d^\ast}\cap {\da} y\subseteq {\ua} F$.

 Assume, on the contrary, that $F_d\cap {\da} y\nsubseteq {\ua} F$ for all $d\in D$. As $\{{\ua} F_{d} : d\in D\}$ is a filtered family, it follows from Claim 2 (or Lemma \ref{lem-filtred-family}) that $\{{\ua} (F_d\cap {\da} y) : d\in D\}$ is a filtered family. Then by Lemma \ref{lem-filtred-family} $\{{\ua} ((F_d\cap {\da} y)\setminus {\ua} F) : d\in D\}\subseteq \mathbf{Fin}~\!\!X$ is a filtered family.
 By Lemma \ref{lem-Rudin-lemma}, there is a directed subset $E\subseteq \mathop{\bigcup}\limits_{d \in D } ((F_d\cap {\da} y)\setminus {\ua} F)$ such that $E\cap ((F_d\cap {\da} y)\setminus {\ua} F)\neq\emptyset$ for all $d\in D$.
  Then ${\ua} y\in \cl_{P_S((X, \tau))}\{{\ua} F_d: d\in D\}\subseteq \Diamond \cl_{\tau}E$, and hence $ y\in \cl_{\tau}E$. As $E\subseteq \mathop{\bigcup}\limits_{d \in D } ((F_d\cap {\da} y)\setminus {\ua} F)$, we have $E\subseteq {\da} y$ and hence ${\da} y= \cl_{\tau}E$.
   By $F\ll_c y$, there is $e\in E$ such that $e\in {\ua} F$, which contradicts $E\subseteq \mathop{\bigcup}\limits_{d \in D } ((F_d\cap {\da} y)\setminus {\ua} F)\subseteq P\setminus {\uparrow} F$. Therefore, there is $d^\ast\in D$ such that $F_d^\ast\cap {\da} y\subseteq {\ua} F$.

   By Claim 3 and Claim 5,  $\mathcal{N}$ is eventually in ${\ua} F$.

(2) $\Rightarrow$ (1): Let $D$ be a  directed subset  of $(X, \tau)$ with ${\da} y=\cl_{\tau}D$. Then $((d)_{d\in D}, y)\in \mathcal{QL}$ by Lemma \ref{lem-direct-set-QL}. By (2) $(d)_{d\in D}$ is eventually in ${\ua} F$, whence $D\cap {\ua} F\neq\emptyset$. Thus $F\ll_c y$.
\end{proof}

\begin{definition}\label{def-WC-space} (\cite[Definition 4.8]{Zhang-Bao-Xu-2022})
 A $T_0$-space $(X,\tau)$ is called a \emph{WC}-\emph{space} if for any $x\in X$, the following three conditions hold:
\begin{enumerate}[\rm (1)]

   \item  ${\Downarrow_c} x $ is directed;

   \item  $ x\in \cl_{\tau}{\Downarrow_c} x$;

   \item  ${\Uparrow_c} x\in \mathcal{O}(\mathcal{L})$.
\end{enumerate}
\end{definition}

\begin{theorem}\label{theor-WC-liminf-topological} (\cite[Theorem 4.11]{Zhang-Bao-Xu-2022})
For a $T_0$-space $(X, \tau)$, the following two conditions are equivalent:
\begin{enumerate}[\rm (1)]

  \item  $(X, \tau )$ is a WC-space.

  \item  The liminf convergence in $(X, \tau)$ is topological, that is, for all $x \in X$ and all nets $\mathcal{N}$ in $X$,
\begin{center}
$\mathcal{N}\stackrel{L}\longrightarrow x$ in $(X, \tau)$ if and only if $\mathcal{N}$ converges to $x$ with respect to $\mathcal{O}(\mathcal{L})$.
\end{center}
\end{enumerate}
\end{theorem}

\begin{corollary}\label{cor-Scott-WC-space}
For a poset $P$, the following two conditions are equivalent:
\begin{enumerate}[\rm (1)]

  \item  $\Sigma P$ is a WC-space.

 \item  The liminf convergence in $\Sigma P$ is topological.
\end{enumerate}
\end{corollary}

\begin{definition}\label{def-WLH-space}
 A $T_0$-space $X$ is called a \emph{weakly locally hypercompact space} (shortly \emph{WLH}-\emph{space}) if for any $x\in X$ and $F\in X^{(<\omega)}$, the following three conditions hold:
\begin{enumerate}[\rm (1)]

   \item   $w_c(x)$ is a filtered subfamily of $\mathbf{Fin}~\!\!X$.

   \item  ${\ua} x\in \cl_{P_S(X)}w_c(x)$.

   \item  $\dua_{c}F\in \mathcal{O}(\mathcal{QL})$.
\end{enumerate}
\end{definition}

\begin{remark}\label{rem-WLH-space} Let $X$ be a $T_0$-space and $x\in X$. If ${\ua} F\in w_c(x)$, then ${\ua} x\subseteq {\ua} F$. Hence by Remark \ref{rem-Smyth-specialization-order}, we have ${\ua} F\in \cl_{P_S(X)}\{{\ua} x\}$. Therefore, $\cl_{P_S(X)}w_c(x)\subseteq \cl_{P_S(X)}\{{\ua} x\}$. So condition (2) of Definition \ref{def-WLH-space} is equivalent to the condition $\cl_{P_S(X)}w_c(x)=\cl_{P_S(X)}\{{\ua} x\}$.
\end{remark}

\begin{proposition}\label{prop-WLH-condition-1-2} Let $X$ be a $T_0$-space and $x \in X$. Then conditions (1) and (2) of Definition \ref{def-WLH-space} hold
if and only if there exists a filtered subfamily $\mathcal{F}\subseteq w_c(x)$ such that $min(F)\subseteq {\da} x$ for any ${\ua} F\in \mathcal{F}$ and ${\ua} x\in \cl_{P_S(X)}\mathcal{F}$ .
\end{proposition}
\begin{proof}
We first prove the necessity. Suppose that conditions (1) and (2) of Definition \ref{def-WLH-space} hold. For any $F\in w_c(x)$, we have $x\in {\ua} F$ or, equivalently, $F\cap {\da} x\neq\emptyset$. Hence by Lemma \ref{lem-filtred-family} and the filteredness of $w_c(x)$, $\mathcal{F}=\{{\ua} (F\cap {\da} x): {\ua} F\in w_c(x)\}$ is filtered. Then by Remark \ref{rem-way-below-c}(5) we have $\mathcal{F}\subseteq w_c(x)$. Clearly, $min(F\cap {\da} x)\subseteq {\da} x$ for any ${\ua} F\in w_c(x)$. It follows from condition (2) of Definition \ref{def-WLH-space} that ${\ua} x\in \cl_{P_S(X)}w_c(x)\subseteq \cl_{P_S(X)}\mathcal{F}$ (in fact, $\cl_{P_S(X)}w_c(x)=\cl_{P_S(X)}\mathcal{F}$).

To show the sufficiency, we assume that there exists a filtered subfamily $\mathcal{F}\subseteq w_c(x)$ such that $min(F)\subseteq {\da} x$ for any ${\ua} F\in \mathcal{F}$ and ${\ua} x\in \cl_{P_S(X)}\mathcal{F}$. Then ${\ua} x\in \cl_{P_S(X)}\mathcal{F}\subseteq \cl_{P_S(X)}w_c(x)$. Now we verify that condition (1) of Definition \ref{def-WLH-space} holds. Let ${\ua} G\in w_c(x)$. We claim that there is ${\ua} F \in \mathcal{F}$ such that ${\ua} F\subseteq {\ua} G$. Assume, on the contrary, that ${\ua} F\nsubseteq {\ua} G$ for all ${\ua} F \in \mathcal{F}$. Then by Lemma \ref{lem-filtred-family} $\{{\ua} min(F\setminus {\ua} G): {\ua} F\in \mathcal{F}\}=\{{\ua} (F\setminus {\ua} G): {\ua} F\in \mathcal{F}\}\subseteq \mathbf{Fin}~\!\!X$ is filtered (note that ${\ua} H={\ua} min(H)$ for any finite subset $H$ of $X$). By Lemma \ref{lem-Rudin-lemma}, there is a directed subset $D\subseteq \bigcup_{{\ua} F\in \mathcal F} min(F\setminus {\ua} G)$ such that $D\cap min(F\setminus {\ua} G)\neq\emptyset$  for all ${\ua} F \in \mathcal{F}$. For any ${\ua} F\in \mathcal F$, it is straightforward to verify that $min(F\setminus {\ua} G)\subseteq min(F)$. Hence $D\subseteq \bigcup_{{\ua} F\in \mathcal F} min(F\setminus {\ua} G)\subseteq \bigcup_{{\ua} F\in \mathcal F} min(F)\subseteq {\da} x$ and $\emptyset\neq D\cap min(F\setminus {\ua} G)\subseteq D\cap min(F)$ for all ${\ua} F\in\mathcal F$. Therefore, ${\ua} x\in \cl_{P_S(X)}\mathcal{F}\subseteq \Diamond \cl_{\mathcal O(X)} D$ and $\cl_{\mathcal O(X)}~\!\! D \subseteq {\da} x$. Thus $\cl_{\mathcal O(X)}~\!\! D ={\da} x$. By $G\ll_{c} x$, we have $D\cap {\ua} G\neq\emptyset$, which contradicts $D\subseteq \bigcup_{{\ua} F\in \mathcal F} min(F\setminus {\ua} G)\subseteq X\setminus {\ua} G$. So there is ${\ua} F \in \mathcal{F}$ such that ${\ua} F\subseteq {\ua} G$. Let ${\ua} G_1, {\ua} G_2\in w_c(x)$. Then there are ${\ua} F_1, {\ua} F_2\in \mathcal{F}$ such that ${\ua} F_1\subseteq {\ua} G_1$  and ${\ua} F_2\subseteq {\ua} G_2$. As $\mathcal{F}$ is filtered, there is $ {\ua} F_3\in \mathcal{F}$ with ${\ua} F_3\subseteq {\ua} F_1\cap{\ua} F_2$. Then ${\ua} F_3\subseteq {\ua} G_1\cap{\ua} G_2$. So $w_c(x)$ is a filtered subfamily of $\mathbf{Fin}~\!\!X$. This completes the proof.
\end{proof}

\begin{remark}
Let $X$ be a $T_0$-space and $x\in X$. Then $\{{\ua} u : u\in {\Downarrow_c} x\}\subseteq w_c(x)$. If $X$ satisfy conditions (1) and (2) of Definition \ref{def-WC-space}, then $\{{\ua} u : u\in {\Downarrow_c} x\}$ is filtered and ${\ua} x\in \cl_{P_S(X)}\{{\ua} u : u\in{\Downarrow_c} x\}$. Therefore, conditions (1) and (2) of Definition \ref{def-WLH-space} are satisfied by Proposition \ref{prop-WLH-condition-1-2}.
\end{remark}

Now we give a characterization of $T_0$-spaces for the quasi-liminf convergence being topological.

\begin{theorem}\label{theor-characterization-WLH-space-1}
For a $T_0$-space $(X, \tau)$, the following two conditions are equivalent:
\begin{enumerate}[\rm (1)]

  \item  $(X, \tau )$ is a WLH-space.

  \item  The quasi-liminf convergence in $(X, \tau)$ is topological, that is, for all $x \in X$ and all nets $\mathcal{N}$ in $X$,
\begin{center}
$\mathcal{N}\stackrel{QL}\longrightarrow x$ in $(X, \tau)$ if and only if $\mathcal{N}$ converges to $x$ with respect to $\mathcal{O}(\mathcal{QL})$.
\end{center}
\end{enumerate}
\end{theorem}

\begin{proof}
(1) $\Rightarrow$ (2): Obviously, $\mathcal{N}\stackrel{QL}\longrightarrow x$ implies $\mathcal{N}$ converges to $x$ with respect to $\mathcal{O}(\mathcal{QL})$. Conversely, suppose that $\mathcal{N}$ converges to $x$ with respect to $\mathcal{O}(\mathcal{QL})$. Now we show that $\mathcal{N}\stackrel{QL}\longrightarrow x$. Firstly, for any ${\ua} F\in w_c(x)$, $\mathcal{N}$ is  eventually in $\dua_{c}F$ by $x\in \dua_{c}F\in \mathcal{O}(\mathcal{QL})$. As $\dua_{c}F\subseteq {\ua} F$, $\mathcal{N}$ is eventually in ${\ua} F$. By (1) and Remark \ref{rem-WLH-space}, we have $\cl_{P_S((X, \tau))}w_c(x)=\cl_{P_S((X, \tau))}\{{\ua} x\}$. Suppose that $G\in X^{(<\omega)}$ and $\mathcal N$ is usually in ${\ua} G$. If $x\not\in {\ua} G$, then $x\in X\setminus {\ua} G\in  \mathcal{O}(\mathcal{QL})$ by Lemma \ref{lem-Lawson-tau-O(QL)}(1). Hence $\mathcal{N}$ is eventually in $X\setminus {\ua} G$, which contradicts the assumption that $\mathcal N$ is usually in ${\ua} G$. So $x\in {\ua} G$. Therefore, $\{F\in X^{(<\omega)} :  F\ll_c x\}$ and $(\mathcal{N}, x)$ satisfy conditions (i)-(iv) of Proposition \ref{prop-quasi-liminf-convergence-charac}(2). Thus $\mathcal{N}\stackrel{QL}\longrightarrow x$ by Proposition \ref{prop-quasi-liminf-convergence-charac}. Thus the quasi-liminf convergence in $(X, \tau)$ is topological.

(2) $\Rightarrow$ (1): Let $x\in X$. We show that $w_c(x)$ is filtered and ${\ua} x\in \cl_{P_S((X, \tau))}w_c(x)$. Let $I=\{(U, n, a): x\in U\in \mathcal{O}(\mathcal{QL}), a\in U, n\in\mathbb{N}\}$ and define an order on $I$ as follows: $(U_1, n_1, a_1)< (U_2, n_2, a_2)$ iff $U_2$ is a proper subset of $U_1$ or $U_2=U_1$ and $n_2>n_1$. Then $I$ is a direct set. Define $x_{(U, n, e)}=e$ for any $(U, n, e)\in I$. It is easy to verify that $(x_i)_{i\in I}$ converges to $x$ with respect to $\mathcal{O}(\mathcal{QL})$. By (2), $(x_i)_{i\in I}\stackrel{QL}\longrightarrow x$, whence there exists a family $\{F_d : d\in D\}\subseteq X^{(<\omega)}$ such that $\{F_d : d\in D\}$ and $((x_i)_{i\in I}, x)$ satisfy conditions (i)-(iv) of Proposition \ref{prop-quasi-liminf-convergence-charac}(2). We claim that $\{{\ua} F_d: d\in D\}\subseteq w_c(x)$. For any $d\in D$, as $F_d$ is a quasi-eventual lower bound of $(x_i)_{i\in I}$, there is $(W, m, c)\in I$ such that $b\in {\ua} F_d$ for any $(V, n, b)>(W, m, c)$. Then for any $t\in W$, we have $(W, m, c)<(W, m+1, t)$, whence $t\in {\ua} F_d$. So $x\in W\subseteq {\ua} F_d$, and hence $x\in \ii_{\mathcal{O}(\mathcal{QL})}{\ua} F_d$. By Remark \ref{rem-way-below-c}(4) we have $F_d\ll_{c} x$. Thus $\{{\ua} F_d: d\in D\}\subseteq w_c(x)$. By Lemma \ref{lem-filtred-family} and Remark \ref{rem-way-below-c}(5), $\{{\ua} (F_d\cap {\da} x) : d\in D\}\subseteq \mathbf{Fin}~\!\!X$ is filtered, $\{{\ua} (F_d\cap {\da} x) : d\in D\}\subseteq w_c(x)$ and ${\ua} x\in \cl_{P_S((X, \tau))} \{{\ua} F_d: d\in D\}\subseteq \cl_{P_S((X, \tau))} \{{\ua} (F_d\cap {\da} x) : d\in D\}$. By Proposition \ref{prop-WLH-condition-1-2}, we have that $w_c(x)$ is filtered and ${\ua} x\in \cl_{P_S((X, \tau))}w_c(x)$. Therefore, $\cl_{P_S((X, \tau))}w_c(x)=\cl_{P_S((X, \tau))}\{{\ua} x\}$ by Remark \ref{rem-WLH-space}.

Now we prove that $\dua_{c}F\in \mathcal{O}(\mathcal{QL})$ for any $F\in X^{(<\omega)}$. Let $y\in \dua_{c}F$ and $((x_i)_{i\in I}, y)\in \mathcal{QL}$. We shall show that $(x_i)_{i\in I}$ is eventually in $\dua_{c}F$. For $i\in I$, let $J(i)=\{(G\cap {\da} x_i, n, e): {\ua} G\in w_c(x_i), e\in  G\cap {\da} x_i, n\in\mathbb{N}\}$ and define an order on $J(i)$ as follows: $(G_1\cap {\da} x_i, n_1, e_1)< (G_2\cap {\da} x_i, n_2, e_2)$ iff ${\ua} (G_2\cap {\da} x_i)$ is a proper subset of ${\ua} (G_1\cap {\da} x_i)$ or ${\ua} (G_2\cap {\da} x_i)={\ua} (G_1\cap {\da} x_i)$ and $n_2>n_1$. By Lemma \ref{lem-filtred-family} and Remark \ref{rem-way-below-c}(5), $\{{\ua} (G\cap {\da} x_i) : {\ua} G\in w_c(x_i)\}\subseteq w_c(x_i)$ is filtered, and hence $J(i)$ is a directed set. For any $j=(G\cap {\da} x_i, n, e)$, define $x_{i, j}=x_{i, (G\cap {\da} x_i, n, e)}=e$. Then we get a net $(x_{i, j})_{j\in J(i)}$. As $\cl_{P_S((X, \tau))}\{{\ua} x_i\}= \cl_{P_S((X, \tau))}w_c(x_i)$, by carrying out a proof similar to that of Lemma \ref{lem-filtered-QL}, we get that $((x_{i, j})_{j\in J(i)}, x_i)\in \mathcal{QL}$. It follows from Theorem \ref{theor-convergence-topological}({\bf Iterated limits}) that $((x_{i,f(i)})_{(i, f)\in I\times M}, y)\in \mathcal{QL}$, where $M=\prod_{i\in I}J(i)=\{f: I\rightarrow \bigcup_{i\in I}J(i): f(i)\in J(i)$ for all $i\in I\}$. Thus $(x_{(i,f(i))})_{(i, f)\in I\times M}$ is eventually in ${\ua} F$ by Lemma \ref{lem-characterization-way-below-c-relation-1}. So there is $(k,g)\in I\times M$ such that $x_{(i,f(i))}\in {\ua} F$ for all $(i, f)> (k,g)$. For each $i^\prime > k$, assume $g(i^\prime)=(G_{i^\prime}\cap {\da} x_{i^\prime}, n^{\prime}, e_{i^\prime})$. For each $s\in G_{i^\prime}\cap {\da} x_{i^\prime}$, define an $f_s\in M$ by

$$f_s (i)=
\begin{cases}
	(G_{i^\prime}\cap {\da} x_{i^\prime}, n^{\prime}+1, s) & i=i^\prime,\\
    g(i)& i\neq i^\prime.
	\end{cases}$$

\noindent Then $(i^\prime, f_s) > (k, g)$, and hence $s=x_{i^\prime, (G_{i^\prime}\cap {\da} x_{i^\prime}, n^{\prime}+1, s)}=x_{i^\prime, f_s(i^\prime)}\in {\ua} F$. Therefore, $x_{i^\prime}\in {\ua} (G_{i^\prime}\cap {\da} x_{i^\prime})\subseteq {\ua} F$. By Remark \ref{rem-way-below-c}(5) we have $x_{i^\prime}\in \dua_c (G_{i^\prime}\cap {\da} x_{i^\prime})\subseteq \dua_c F$ for all $i^\prime > k$. Hence $(x_i)_{i\in I}$ is eventually in $\dua_{c}F$. So $\dua_{c}F\in \mathcal{O}(\mathcal{QL})$. This completes the proof that $(X, \tau )$ is a \emph{WLH}-space.
\end{proof}

\begin{corollary}\label{cor-characterization-Scott-WLH-space-1}
For a poset $P$, the following two conditions are equivalent:
\begin{enumerate}[\rm (1)]

  \item  $\Sigma~\!\! P$ is a WLH-space.

 \item  The quasi-liminf convergence in $\Sigma~\!\! P$ is topological, that is, for all $x \in P$ and all nets $\mathcal{N}$ in $P$,
\begin{center}
$\mathcal{N}\stackrel{QL}\longrightarrow x$ in $\Sigma~\!\!P$ if and only if $\mathcal{N}$ converges to $x$ with respect to $\mathcal{O}(\mathcal{QL})$.
\end{center}
\end{enumerate}
\end{corollary}

\begin{theorem}\label{theor-LH-is-WLH}
Every locally hypercompact $T_0$-space is a WLH-space.
\end{theorem}

\begin{proof} Suppose that $(X, \tau)$ is a locally hypercompact $T_0$-space. We will show that it is a WLH-space. For $x\in X$, let $\mathcal F_x=\{{\ua} (F\cap {\da} x) : F\in X^{(<\omega)}$, $x\in \ii_{\tau}{\ua} F\}$.

{\bf Claim 1:} For any $G\in X^{(<\omega)}$ with $x\in \ii_{\tau}{\ua} G$, $x\not \in {\ua} (G\setminus (G\cap {\da} x))$.

If $x\in {\ua} (G\setminus (G\cap {\da} x))$, then there is $u\in G\setminus (G\cap {\da} x)$ with $u\leq x$. Then $u\in G\cap {\da}x$, which contradicts $u\in G\setminus (G\cap {\da} x)$. Hence $x\not \in {\ua} (G\setminus (G\cap {\da} x))$.

{\bf Claim 2:} $\mathcal F_x$ is a filtered subfamily of $w_c(x)$ and $\cl_{P_S((X, \tau))}\mathcal F_x=\cl_{P_S((X, \tau))}\{{\ua} x\}$.

 For any $H\in X^{(<\omega)}$ with $x\in\ii_\tau {\ua} H$, we have $H\ll_c x$ and hence $H\cap {\da} x\ll_c x$ by Remark \ref{rem-way-below-c}(5). So $\mathcal F_x\subseteq w_c(x)$.  As $(X, \tau)$ is locally hypercompact, $\{{\ua} F : F\in X^{(<\omega)}$, $x\in \ii_{\tau}{\ua} F\}$ is filtered. It follows from Lemma \ref{lem-filtred-family} that $\mathcal F_x=\{{\ua} (F\cap {\da} x) :  F\in X^{(<\omega)}$, $x\in \ii_{\tau}{\ua} F\}$ is filtered. By the local hypercompactness of $(X, \tau)$, we have that $\cl_{P_S((X, \tau))}\{{\ua} F: F\in X^{(<\omega)}$, $x\in \ii_{\tau}{\ua} F\}=\cl_{P_S((X, \tau))}\{{\ua} x\}$. Hence ${\ua} x\in \cl_{P_S((X, \tau))}\{{\ua} F: F\in X^{(<\omega)}$, $x\in \ii_{\tau}{\ua} F\}\subseteq \cl_{P_S((X, \tau))}\{{\ua} (F\cap {\da} x) : F\in X^{(<\omega)}$, $x\in \ii_{\tau}{\ua} F\}\subseteq \cl_{P_S((X, \tau))}\{{\ua} x\}$ (note that ${\ua} x\subseteq {\ua} (F\cap {\da} x)\subseteq {\ua} F$ for any $F\in X^{(<\omega)}$ with $x\in\ii_\tau {\ua} F$). Thus $\cl_{P_S((X, \tau))}\mathcal F_x=\cl_{P_S((X, \tau))}\{{\ua} x\}$.

{\bf Claim 3:} $w_c(x)$ is filtered and $\cl_{P_S((X, \tau))}w_c(x)=\cl_{P_S((X, \tau))}\{{\ua} x\}$.

By Claim 2 and Proposition \ref{prop-WLH-condition-1-2}, we have that $w_c(x)$ is filtered  and $\cl_{P_S((X, \tau))}w_c(x)=\cl_{P_S((X, \tau))}\{{\ua} x\}$.

{\bf Claim 4:}  For any ${\ua} F\in \mathbf{Fin}X$ and $U\in \omega(X)\vee \tau$ with $U\subseteq {\ua} F$, $U\subseteq \dua_{c}F$.

Suppose that $u\in U$. Then there exist $V\in \tau$ and $G\in X^{(<\omega)}$ such that $u\in V\setminus {\ua} G= V\cap (X\setminus {\ua} G)\subseteq U$. Let $D$ be a directed subset in $(X, \tau)$ with ${\da} u=\cl_{\tau}D$. Then $D\subseteq {\da} u\subseteq X\setminus {\ua} G$. As $u\in V\cap\cl_{\tau}D$, we have $D\cap V\neq\emptyset$, whence $\emptyset\neq D\cap V=D\cap (X\setminus {\ua} G)\cap V=D\cap (V\setminus {\ua} G)\subseteq D\cap U\subseteq D\cap {\ua} F$. Hence $u\in \dua_{c}F$. Thus $U\subseteq \dua_{c}F$.

{\bf Claim 5:} For any $F\in X^{(<\omega)}$, $\dua_{c}F\in \mathcal{O}(\mathcal{QL})$.

We first show $\dua_{c}F\in \omega(X)\vee \tau$. Let $z\in \dua_{c}F$ and $\{G_j : j\in J\}=\{G\in X^{<\omega} : z\in \ii_{\tau}{\ua} G\}$. Then $\mathcal F_z=\{{\ua} (G_j\cap {\da} z) : j\in J\}$.  We claim that there is $i\in J$ such that ${\ua} (G_{i}\cap {\da} z)\subseteq {\ua} F$. Assume, on the contrary, that ${\ua} (G_j\cap {\da} z)\nsubseteq {\ua} F$ for all $j\in J$. Then by Claim 2 and Lemma \ref{lem-filtred-family}(2), $\{{\ua} ((G_j\cap {\da} z)\setminus {\ua} F) : j\in J\}$ is a filtered subfamily of $\mathbf{Fin} X$. By Lemma \ref{lem-Rudin-lemma}, there is a directed subset $D\subseteq \bigcup_{j\in J} ((G_j\cap {\da} z)\setminus {\ua} F)$ such that $D\cap ((G_j\cap {\da} z)\setminus {\ua} F)\neq\emptyset$  for all $j\in J$. For any $j\in J$, we have $(G_j\cap {\da} z)\setminus {\ua} F\subseteq G_j\cap {\da} z\subseteq {\da} z$, whence $D\subseteq \bigcup_{j\in J} ((G_j\cap{\da} z) \setminus {\ua} F)\subseteq {\da} z$ and $\emptyset\neq D\cap ((G_j\cap {\da} z)\setminus {\ua} F)\subseteq D\cap (G_j\cap {\da} z)$. Therefore, by Claim 1, ${\ua} z\in \cl_{P_S(X)}\mathcal{F}_z\subseteq \Diamond \cl_{\tau} D$ and $\cl_{\tau}~\!\! D \subseteq {\da} z$. Thus $\cl_{\tau}~\!\! D ={\da} z$. By $F\ll_{c} z$, we have $D\cap {\ua} F\neq\emptyset$, which contradicts $D\subseteq \bigcup_{j\in J} ((G_j\cap {\da} z)\setminus {\ua} F)\subseteq X\setminus {\ua} F$.
 So there is $i\in J$ such that ${\ua} (G_i\cap {\da} z)\subseteq {\ua} F$.
 Then by Claim 1,  $z\in \ii_{\tau}{\ua} G_i\setminus {\ua} (G_i\setminus (G_i\cap {\da} z))\subseteq{\ua} G_i\setminus {\ua} (G_i\setminus (G_i\cap {\da} z)) \subseteq {\ua} (G_i\cap {\da} z)\subseteq {\ua} F$.
 By Claim 4, we have $z\in \ii_{\tau}{\ua} G_i\setminus {\ua} (G_i\setminus (G_i\cap {\da} z))\subseteq \dua_c F$. Hence $\dua_{c}F\in \omega(X)\vee \tau$. By Lemma \ref{lem-Lawson-tau-O(QL)}(2), $\dua_{c}F\in \mathcal{O}(\mathcal{QL})$.

By Claims 3 and 5, $(X, \tau )$ is a \emph{WLH}-space.
\end{proof}

By Lemma \ref{lem-Lawson-tau-O(QL)}(2), Theorem \ref{theor-characterization-WLH-space-1} and  Theorem \ref{theor-LH-is-WLH}, we obtain the following.

\begin{corollary}\label{cor-LHC-quasi-liminf-Kou}  (\cite[Theorem 4.5]{Chen-Kou-2024})
For a locally hypercompact $T_0$-space $(X, \tau)$, the quasi-liminf convergence (i.e., the $\mathcal{QL}$-convergence) is topological and agrees with convergence in $\omega(X)\vee \tau$.
\end{corollary}

From Proposition \ref{prop-Lawson-tau-O(QL)-O(L)}(2) and Corollary \ref{cor-LHC-quasi-liminf-Kou} we deduce the following.

\begin{corollary}\label{cor-quasicontinuous-quas-liminf}
For a quasicontinuous poset $P$, the quasi-liminf convergence is topological and agrees with convergence in the Lawson topology $\lambda(P)$.
\end{corollary}

By Proposition \ref{prop-quasicontinuous-local-hypercompact} and Corollary \ref{cor-quasicontinuous-quas-liminf}, we get the following.

\begin{corollary}\label{cor-quasicontinuous-domain-quas-liminf} (\cite[Theorems 4.1 and 4.4]{Zhou-Li-2013})
For a quasicontinuous domain $P$, the quasi-liminf convergence is topological and agrees with convergence in the Lawson topology $\lambda(P)$.
\end{corollary}

The following example shows that a weakly locally hypercompact $T_0$-space may not be locally hypercompact.

\begin{example}\label{exmp-WLH-not-LH}
Let $P=\{a_n: n\in \mathbb{N}\}\cup\{b_n: n\in \mathbb{N}^+\}\cup\{c_i: i\in \mathbb{N}^+\}$. Define an order on $P$ as follows (see Figure 3):

\begin{enumerate}[\rm (1)]
 \item  $a_n\leq a_m$ for all $m, n\in \mathbb{N}^+$ with $n\leq m$;
 \item  $b_n\leq b_m$ for all $m, n\in \mathbb{N}^+$ with $n\leq m$;
 \item  $a_n< a_0$ and $b_n< a_0$ for all $n\in \mathbb{N}^+$; and
 \item  $c_n$ and $c_m$ are incomparable for all $m, n\in \mathbb{N}^+$ with $n\leq m$.

\end{enumerate}

\begin{figure}[h]
    \centering
    \includegraphics[width=10cm, height=4cm, keepaspectratio]{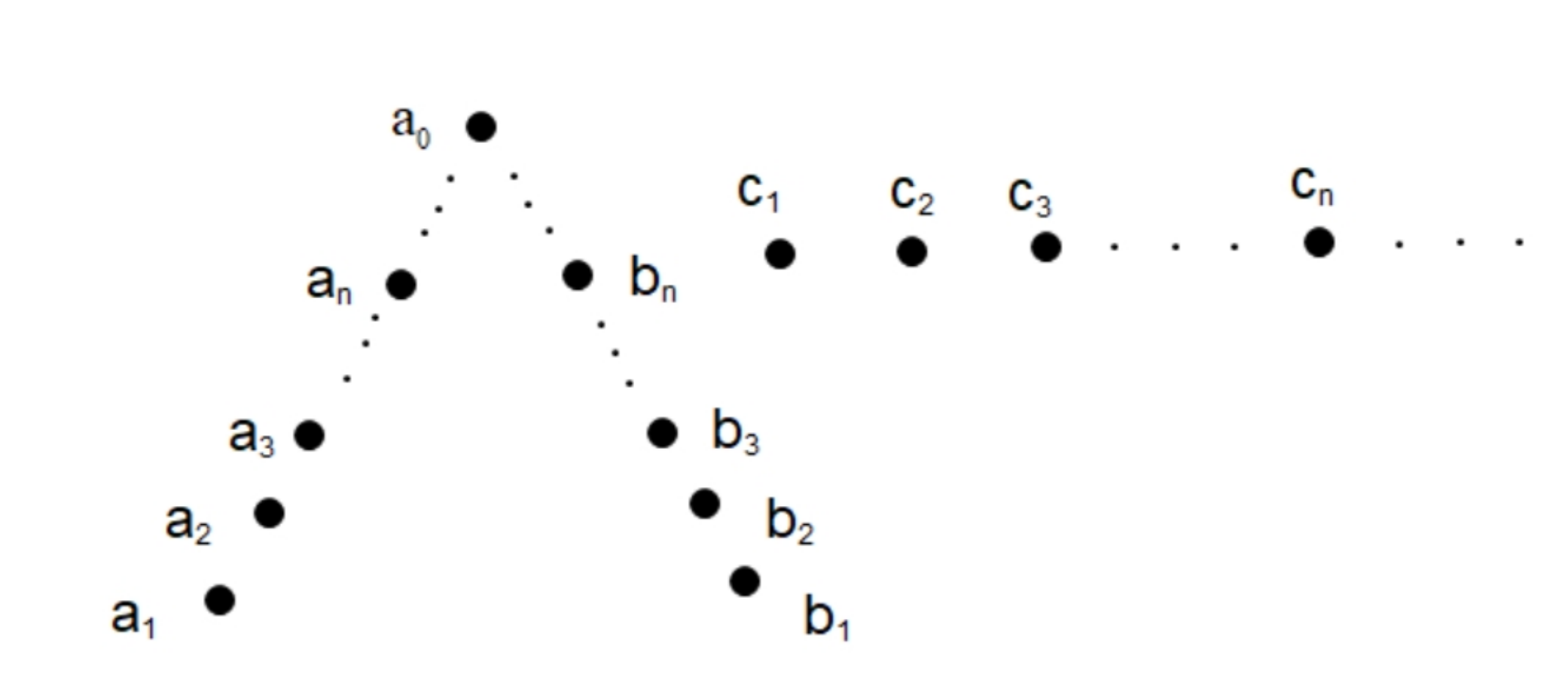}
    \caption{The countable quasicontinuous domain $P$ in Example \ref{exmp-WLH-not-LH}.}
\end{figure}

Consider the upper topology $\upsilon(P)$ on $P$. Then we have the following conclusions:

\begin{enumerate}[\rm (a)]
\item $\mathrm{Id}(P)=\{{\da} x: x\in P\}\cup\{\{a_n: n\in \mathbb{N}^+\},\{b_n: n\in \mathbb{N}^+\}\}$. Hence $P$ is a dcpo.

\item  $(P, \upsilon(P))$ is not locally hypercompact.

In fact, for any $G, H\in P^{(<\omega)}$, $(P\setminus {\da} G)\cap\{c_n: n\in \mathbb{N}^+\}=(P\setminus {\da} G)\cap (P\setminus {\da}a_0)\in\upsilon(P)$ is infinite and $\{c_n: n\in \mathbb{N}^+\}\cap {\uparrow} H$ is finite. Hence $\ii_{\upsilon(P)}{\ua} F=\emptyset$ for any $F\in P^{(<\omega)}$. Therefore, $(P, \upsilon(P))$ is not locally hypercompact.

\item $P$ is a quasicontinuous domain but is not a continuous domain.

Since ${\Downarrow} a_0=\{u\in P : u\ll a_0\}=\emptyset$, $P$ is not a continuous domain. Now we show that $P$ is a quasicontinuous domain.

\begin{enumerate}[\rm (i)]
 \item   If $x\in \{c_n:n\in \mathbb{N}^{+}\}$, then $x\ll x$.

  \item  If $x\in \{a_n:n\in \mathbb{N}^{+}\}$, then $\{{\ua}\{x,b_m\}:m\in \mathbb{N}^{+}\}$ is filtered with $\{{\ua}\{x, b_m\}:m\in \mathbb{N}^{+}\}\subseteq w(x)$ and $\bigcap_{m\in \mathbb{N}^{+}}{\ua}\{x, b_m\}={\ua} x$.

 \item   If $x\in \{b_n:n\in \mathbb{N}^{+}\}$, then $\{{\ua}\{x,a_m\}:m\in \mathbb{N}^{+}\}$ is filtered with $\{{\ua}\{x, a_m\}:m\in \mathbb{N}^{+}\}\subseteq w(x)$ and $\bigcap_{m\in \mathbb{N}^+}{\ua}\{x, a_m\}={\ua} x$.

    \item   If $x=a_0$, then $\{{\ua}\{a_n, b_n\}:n\in \mathbb{N}^{+}\}$ is filtered with $\{{\ua}\{a_n, b_n\}:n\in \mathbb{N}^{+}\}\subseteq w(x)$ and $\bigcap_{n\in \mathbb{N}^+}{\ua}\{a_n, b_n\}={\ua} a_0$.
\end{enumerate}

By (i)-(iv) and Proposition \ref{prop-quasicontinuous-sufficient-condition}, $P$ is a quasicontinuous domain.

\item  By (a) we have $\ll_c=\ll$ (i.e., for $x\in P$ and $F\in P^{(<\omega)}$, $F\ll_c x$ iff $F\ll x$).

\item  For any $x\in P$ and $F\in P^{(<\omega)}$, $w_c(x)$ is filtered  and $\dua_{c}F\in \sigma(P)$.

By (c) and (d), $w_c(x)=w(x)$ is filtered and ${\ua} x=\bigcap w_c(x)$, and by \cite[Proposition 4.1]{Gierz-Lawson-Stralka-1983} $\dua_{c}F=\dua F\in \sigma(P)$.

\item For any $F\in X^{(<\omega)}$, $\dua_{c}F\in \mathcal{O}(\mathcal{QL})$.

By (e), Proposition \ref{prop-md-upper-topology} and Proposition \ref{prop-Lawson-tau-O(QL)-O(L)}(1), we have $\dua_{c}F\in \sigma(P)=md(\upsilon(P))\subseteq \mathcal{O}(\mathcal{QL})$.

\item  For any $x\in P$, ${\ua} x\in \cl_{P_S(X)}w_c(x)$.

Let $U\in \upsilon(P)$ with $x\in U$. Then we have the following conclusions:

\begin{enumerate}[\rm (i)]
\item If $x\in \{c_n: n\in \mathbb{N}^+\}$, then $x\ll_c x$, that is, ${\ua} x=\{x\}\in w_c(x)$;

\item If $x\in \{a_n: n\in \mathbb{N}^+\}$, then by $\bigvee_{n\in \mathbb{N}^+}b_n=a_0\in {\ua} x\subseteq U\in \upsilon(P)\subseteq \sigma(P)$, there exists $m\in \mathbb{N}^+$ such that $b_m\in U$ and hence ${\ua}\{x, b_m\}\subseteq U$. Clearly,  ${\ua}\{x, b_m\}\in w(x)=w_c(x)$.

\item If $x\in \{b_n: n\in \mathbb{N}^+\}$, then by $\bigvee_{n\in \mathbb{N}^+}a_n=a_0\in {\ua} x\subseteq U\in \upsilon(P)\subseteq \sigma(P)$, there exists $l\in \mathbb{N}^+$ such that $a_l\in U$ and hence ${\ua}\{x, a_l\}\subseteq U$. Clearly,  ${\ua}\{x, a_l\}\in w(x)=w_c(x)$.

\item If $x=a_0$, then  by $\bigvee_{n\in \mathbb{N}^+}a_n=\bigvee_{n\in \mathbb{N}^+}b_n=a_0\in {\ua} x\subseteq U\in \upsilon(P)\subseteq \sigma(P)$, there exists $m_1, m_2\in \mathbb{N}^+$ such that $\{a_{m_1}, b_{m_2}\}\subseteq U$, and hence ${\ua}\{a_{m_1}, b_{m_2}\}\subseteq U$. Clearly, ${\ua} \{a_{m_1}, b_{m_2}\}\in w(x)=w_c(x)$;
\end{enumerate}

It follows from (i)-(iv) that $\Box U\cap w_c(x)\neq\emptyset$. Thus ${\ua} x\in \cl_{P_S(X)}w_c(x)$.

 \end{enumerate}

By (e), (f) and (g), $(P, \upsilon(P))$ is weakly locally hypercompact.
\end{example}

Finally, using the quasi-liminf convergence and approximate relation $\ll_c$, we give some characterizations of $C$-spaces.

\begin{lemma}\label{lem-GWC-space} Let $(X, \tau)$ be a $T_0$ space and $H\subseteq X$. Then the following conditions are equivalent:
\begin{enumerate}[\rm (1)]
\item $\dua_c H\in \tau$.
\item $\dua_c H=\ii_\tau {\ua} H$.
\item For any $x\in X$, $H\ll_c x$ iff $x\in \ii_{\tau} {\ua} H$.
\end{enumerate}
\end{lemma}
\begin{proof} (1) $\Rightarrow$ (2): Suppose that $\dua_c H\in \tau$. Then by Lemma \ref{lem-Lawson-tau-O(QL)}(1) and Remark \ref{rem-way-below-c}(4), we have $\ii_\tau {\ua} H\subseteq \ii_{\mathcal{O}(\mathcal{QL})} {\ua} H\subseteq \dua_c H\subseteq {\ua} H$. As $\dua_c H\in \tau$, we have $\dua_c H\subseteq \ii_\tau {\ua} H$. Hence $\dua_c H=\ii_{\tau} {\ua} H$.

(2) $\Rightarrow$ (1), (2) $\Leftrightarrow$ (3): Trivial.
\end{proof}

\begin{theorem}\label{theor-characterization-C-space-1} For a $T_0$-space $(X, \tau)$, the following conditions are equivalent:
\begin{enumerate}[\rm (1)]
\item $(X, \tau)$ is a $C$-space.
\item For any $U\in \tau$ and $x\in U$, there exists $F \in X^{(<\omega)}$ such that $F\subseteq {\da} x$ and $x\in \ii_{\tau} {\ua} F\subseteq  {\ua} F\subseteq U$.
\item $(X, \tau)$ is locally hypercompact and for any $(x, F)\in X\times X^{(<\omega)}$, $x\in  \ii_{\tau} {\ua} F$ implies  $x\in  \ii_{\tau} {\ua} (F\cap {\da} x)$.
\item $(X, \tau)$ is locally hypercompact and $\dua_c F\in \tau$ for all $F\in X^{(<\omega)}$.
\item $(X, \tau)$ is a WLH-space and $\dua_c F\in \tau$ for all $F\in X^{(<\omega)}$.

\end{enumerate}
\end{theorem}
\begin{proof} (1) $\Rightarrow$ (2): Suppose that $U\in \tau$ and $x\in U$. As $(X, \tau)$ is  a $C$-space, there exists $u \in X$ such that $x\in \ii_{\tau} {\ua} u\subseteq {\ua} u\subseteq U$. Let $F=\{u\}$. Then $F\subseteq {\da} x$ and $x\in \ii_{\tau} {\ua} F\subseteq {\ua} F\subseteq U$.

(2) $\Rightarrow$ (3): Assume that condition (2) holds. Clearly, $(X, \tau)$ is locally hypercompact. Suppose that $(x, F)\in X\times X^{(<\omega)}$ satisfying $x\in  \ii_{\tau} {\ua} F$. By (2), there is $G\in X^{(<\omega)}$ such that $G\subseteq {\da} x$ and $x\in \ii_{\tau} {\ua} G\subseteq  {\ua} G\subseteq \ii_{\tau}{\ua} F\subseteq {\ua} F$. Then ${\ua} G\subseteq {\ua} ({\ua} F\cap {\da} x)={\ua} (F\cap {\da} x)$. Hence  $x\in \ii_{\tau} {\ua} G\subseteq \ii_{\tau}{\ua} (F\cap {\da} x)$.

(3) $\Rightarrow$ (1): For any $x\in X$ and $U\in\mathcal O(X)$, we will show that ${\ua} (U\cap {\da} x)\in \tau$. Let $y\in {\ua} (U\cap {\da} x)$. Then there is a $u\in U\cap {\da} x$ such that $u\leq y$. As $(X, \tau)$ is locally hypercompact, there exists $F \in X^{(<\omega)}$ such that $u\in \ii_{\tau}{\ua} F\subseteq {\ua} F\subseteq U$. By (3), we have $u\in \ii_{\tau}{\ua} (F\cap {\da} u)$. Then $y\in \ii_{\tau}{\ua} (F\cap {\da} u)\subseteq {\ua} (U\cap {\da} u)\subseteq {\ua} (U\cap {\da} x)$. Hence ${\ua} (U\cap {\da} x)\in \tau$. By Lemma \ref{lem-closed-set-lattice-Heyting} and Proposition \ref{prop-C-space-LHC-web}, $(X, \tau)$ is a $C$-space.

(1) $\Rightarrow$ (4): Suppose that $(X, \tau)$ is a $C$-space and $F\in X^{(<\omega)}$. Then $(X, \tau)$ is locally hypercompact. For any $x\in \dua_c F$, let $D_x=\{d\in X : x\in \ii_{\tau}{\ua} d\}$. As $(X, \tau)$ is a $C$-space, $D_x$ is directed and $\cl_{\tau} D_x={\da} x$. By $F\ll_c x$, there is $d\in D_x$ such that $d\in {\ua} F$. Hence $x\in \ii_{\tau}{\ua} d\subseteq \ii_{\tau}{\ua} F\subseteq \ii_{\mathcal{O}(\mathcal{QL})}{\ua} F\subseteq \dua_c F$ by Proposition \ref{prop-Lawson-tau-O(QL)-O(L)}(1) and Remark \ref{rem-way-below-c}(4). Thus $\dua_c F\in \tau$.

(4) $\Rightarrow$ (5): By Theorem \ref{theor-LH-is-WLH}.

(5) $\Rightarrow$ (2): For $U\in\mathcal O(X)$ and $x\in U$, let $\mathcal F_x=\{{\ua} G : G\in X^{(<\omega)}, x\in \ii_{\tau}{\ua} G\}$. Then $w_c(x)=\mathcal F_x$ by (5) and Lemma \ref{lem-GWC-space}. As $(X, \tau)$ is a WLH-space, we have that $\mathcal F_x$ is filtered and ${\ua} x\in\cl_{P_S(X)}\mathcal F_x$. By ${\ua} x\in \Box U\in \mathcal O(P_S(X))$, there is ${\ua} H\in \mathcal F_x$ such that ${\ua} H\in \Box U$ or, equivalently, ${\ua} H\subseteq U$. Hence $x\in \ii_{\tau} {\ua} H\subseteq {\ua} H\subseteq U$. Then $H\ll_c x$. By Remark \ref{rem-way-below-c}(5), we have $H\cap {\da} x\ll_c x$. It follows from (5) and Lemma \ref{lem-GWC-space} that $x\in \ii_{\tau}{\ua} (H\cap {\da} x)\subseteq {\ua} (H\cap {\da} x)\subseteq {\ua} H\subseteq U$. Let $F=H\cap {\da} x$. Then $F \in X^{(<\omega)}$, $F\subseteq {\da} x$ and $x\in \ii_{\tau} {\ua} F\subseteq  {\ua} F\subseteq U$.
\end{proof}

\begin{theorem}\label{theor-characterization-C-space-2} For a $T_0$-space $(X, \tau)$, the following conditions are equivalent:
\begin{enumerate}[\rm (1)]
\item $(X, \tau)$ is a $C$-space.
\item $\dua_c F\in \tau$ for all $F\in X^{(<\omega)}$, and the quasi-liminf convergence coincides with convergence in the topology $\omega(X)\vee \tau$, that is, for all $x\in X$ and all nets $(x_i)_{i\in I}$ in $X$,
\begin{center}
$(x_i)_{i\in I}$  $\stackrel{QL}\longrightarrow x$ if and only if $(x_i)_{i\in I}$ converges to $x$ with respect to the topology $\omega(X)\vee \tau$.
\end{center}\item The quasi-liminf convergence in $(X, \tau)$ is topological and $\dua_c F\in \tau$ for all $F\in X^{(<\omega)}$.

\end{enumerate}
\end{theorem}
\begin{proof} (1) $\Leftrightarrow$ (2): By Lemma \ref{lem-Lawson-tau-O(QL)}(2), Theorem \ref{theor-characterization-WLH-space-1} and Theorem \ref{theor-characterization-C-space-1}.

(2) $\Rightarrow$ (3): Trivial.

(3) $\Rightarrow$ (1): By Theorem \ref{theor-characterization-WLH-space-1} and Theorem \ref{theor-characterization-C-space-1}.


\end{proof}

By Proposition \ref{prop-continuous-c-space}, Proposition \ref{prop-quasicontinuous-local-hypercompact}, Theorem \ref{theor-characterization-C-space-1}  and Theorem \ref{theor-characterization-C-space-2}, we get the following.

\begin{corollary}\label{cor-characterization-continuous-poset-1} For a poset $P$, the following conditions are equivalent:
\begin{enumerate}[\rm (1)]
\item $P$ is continuous.
\item For any $U\in \sigma(P)$ and $x\in U$, there exists $F \in P^{(<\omega)}$ such that $F\subseteq {\da} x$ and $x\in \ii_{\sigma(P)} {\ua} F\subseteq  {\ua} F\subseteq U$.
\item $P$ is quasicontinuous and for any $(x, F)\in P\times P^{(<\omega)}$, $x\in  \ii_{\sigma(P)} {\ua} F$ implies  $x\in  \ii_{\sigma(P)} {\ua} (F\cap {\da} x)$.
\item $P$ is quasicontinuous and $\dua_c F\in \sigma(P)$ for all $F\in P^{(<\omega)}$.
\item $\Sigma P$ is a WLH-space and $\dua_c F\in \sigma(P)$ for all $F\in P^{(<\omega)}$.
\item The quasi-liminf convergence in $\Sigma P$ is topological and $\dua_c F\in \sigma(P)$ for all $F\in P^{(<\omega)}$.
\item The quasi-liminf convergence coincides with convergence in the Lawson topology $\lambda(P)$, and $\dua_c F\in \sigma(P)$ for all $F\in P^{(<\omega)}$.
\end{enumerate}
\end{corollary}






\end{document}